\pdfoutput=1 
\RequirePackage{fix-cm}
\PassOptionsToPackage{fixpdftex,table,hyperref,xcdraw,dvipsnames,x11names,svgnames*}{xcolor}
\PassOptionsToPackage{usenames}{color}
\documentclass{ws-jml}
\usepackage{fixltx2e}
\usepackage[utf8]{inputenc}
\RequirePackage{xcolor}
\usepackage{comment}
\usepackage{ifpdf}
\usepackage{fancybox}
\usepackage{hyperref}
\hypersetup{		pdfborder=0 0 0,%
								colorlinks=true,%
								linkbordercolor={0 0 0},
								hyperfootnotes=true,%
								pdftex=true, %
								implicit=true,%
								final=true,%
								letterpaper=true, %
								bookmarks=true,                       
                bookmarksnumbered=true,               
                bookmarksopen=true,							%
								pdfpagemode=UseOutlines,                
	              pdfstartview=Fit,                
	              pdfpagelayout=OneColumn,            
	              breaklinks=true,                 
	              pageanchor=true,                 
	              plainpages=false,                
	              pdfpagelabels=true,                   
	              backref=page,%
	              pagebackref=true,
	              hyperindex=true,                      
								pdfdisplaydoctitle=true,%
								pdfauthor={Peter M. Gerdes},%
pdfkeywords={Implicit Definability;\( \Pi^0_1 \) Classes;Fast Growing Functions;Turing Degrees},%
pdfsubject={MSC2010: Primary 03D55, 03D60},%
pdftitle={Harrington's Solution to McLaughlin's Conjecture and Non-uniform Self-moduli}}
\ifpdf
	\RequirePackage{cmap} 
	\RequirePackage{hyperxmp}
	\RequirePackage{hyperref}
\else
\fi
\RequirePackage{amsmath}
\RequirePackage{amssymb}
\RequirePackage{bm}
\RequirePackage{fixmath}
\RequirePackage{undertilde}
\RequirePackage{array}
\RequirePackage{centernot} 
\RequirePackage[disallowspaces,fixamsmath]{mathtools}
\mathtoolsset{showonlyrefs,showmanualtags}
\RequirePackage{marvosym}
\RequirePackage{amsfonts}
\RequirePackage{stmaryrd}
\DeclareMathAlphabet{\mathbrush}{T1}{pbsi}{xl}{n}
\RequirePackage{mathrsfs}

\RequirePackage{euscript}

\ifpdf
	\RequirePackage{tikz}
	\RequirePackage{pgfbaseplot}
	\RequirePackage{pgffor}
	\usepgflibrary{plothandlers}
\else
	\usepackage{graphicx}
\fi
\RequirePackage{color}
\makeatletter
\@ifundefined{plmac@tag}{}{\usepackage{perltex}}
\makeatother

\makeatletter

\newsavebox{\@PMG@fmbox} 
\newcommand{\fmpage}[2]{\begin{lrbox}{\@PMG@fmbox}\begin{minipage}{#1} #2 \end{minipage}\end{lrbox}\fbox{\usebox{\@PMG@fmbox}}}

\makeatother
\RequirePackage{suffix}
\RequirePackage{ifmtarg}
\RequirePackage{xifthen}

\makeatletter

\newcommand*{\norm}[1]{\lVert #1 \rVert}

\let\seq=\sequence

\usepackage{rec-thy}

\newcommand*{\funl}{\ll}
\newcommand*{\nfunl}{\not\ll}
\newcommand*{\fung}{\gg}
\newcommand*{\nfung}{\not\gg}

\WithSuffix\def\fung*{\fung^{*}}
\WithSuffix\def\funl*{\vphantom{\funl}^{*}\funl}
\WithSuffix\def\nfung*{\nfung^{*}}
\WithSuffix\def\nfunl*{\vphantom{\nfunl}^{*}\nfunl}

\newcommand{\axiom}[2]{ \mathopen{\langle} \ifthenelse{\isempty{#1}}{\eset}{#1} \rightarrow #2 \mathclose{\rangle} }

\newcommand*{\lang}{\mathscr{L}}
\WithSuffix\def\lang[#1]{\lang\left(#1\right)}

\newcommand*{\treeroot}[1][]{0_T}

\newcommand*{\treeRoot}[1]{\ifthenelse{\isempty{#1}}{\eset}{0_{T}}}

\makeatother

\newcommand*{\enumseq}[1]{#1^{\blacktriangleleft}}
\WithSuffix\def\enumseq[#1]#2{#2\triangleleft #1}
\newcommand*{\copylen}[2][]{\mathbrush{l}^{#1}(#2)}
\newcommand*{\enumpred}[1]{{#1}^{\diamond}}

\newcommand*{\treemap}[2]{\vartheta^{#1}_{#2}}

\newcommand*{\nusm}{\zeta}
\makeatletter
\newcommand*{\fastn}{\let\@PMG@parenarg\@PMG@undefined\let\@PMG@braketarg\@PMG@undefined\@fastnbody}
\newcommand*{\@fastnbody}[1]{\xi^{#1\ifdefined\@PMG@parenarg
, \@PMG@parenarg%
\fi%
}
\ifdefined\@PMG@braketarg
_{\@PMG@braketarg}%
\fi}
\WithSuffix\def\@fastnbody(#1){\def\@PMG@parenarg{#1}\@fastnbody}
\WithSuffix\def\@fastnbody[#1]{\def\@PMG@braketarg{#1}\@fastnbody}
\makeatother

\begin{document}

\title{Harrington's Solution to McLaughlin's Conjecture and Non-uniform Self-moduli\footnote{Preprint of an article submitted for consideration in \href{http://www.worldscinet.com/jml/}{Journal of Mathematical Logic} in 2010.  On acceptance copyright will be assigned to World Scientific Publishing Company}}
\author{Peter M. Gerdes\footnote{Partially Supported by NSF EMSW21-RTG-0739007 and EMSW21-RTG-0838506}}
\address{Department of Mathematics\\
University of Notre Dame du Lac\\
Notre Dame, Indiana 46556}

\maketitle

\begin{abstract}
	While much work has been done to characterize the Turing degrees computing members of various collections of fast growing functions, much less has been done to characterize the rate of growth necessary to compute particular degrees.  Prior work has shown that every degree computed by all sufficiently fast growing functions is uniformly computed by all sufficiently fast growing functions.  We show that the rate of growth sufficient for a function to uniformly compute a given Turing degree can be separated by an arbitrary number of jumps from the rate of growth that suffices for a function to non-uniformly compute the degree.  These results use the unpublished method Harrington developed to answer McLaughlin's conjecture so we begin the paper with a rigorous presentation of the approach Harrington sketched in his handwritten notes on the conjecture.  We also provide proofs for the important computability theoretic results Harrington noted were corollaries of this approach.  In particular we provide the first published proof of Harrington's result that there is an effectively given sequence of \( \pizn{1} \) singletons that are \( Low_\alpha \)  none of which is computable in the effective join of the \( \alpha \) jumps of the others for every \( \alpha \kleenel \wck \).
\end{abstract}


\section{Introduction}

\subsection{Remarks}

While this paper was drafted to convey a result of the author's the first half of this paper is devoted to the presentation of Harrington's results from \cite{mclaughlins-conjecture} as the technique he used to settle McLaughlin's conjecture is needed for the author's result and has never before appeared in print.  The author would like to make absolutely clear that these results are Harrington's alone, but as Harrington's notes are quite sparse and the way in which he anticipated filling in the details has faded with time, the details are, for good or ill, the author's own.  In addition to various assorted details the technical results in \ref{sec:ord-org} on nice ordinal notations and the modifications required to prove lemma \ref{lem:subgeneric-root} and corollary \ref{cor:subgeneric-star} true are of the author's devising and it is unclear what it any resemblance they might bear to Harrington's original conception of these proofs.  Once Harrington's method has been presented the second half of the paper will revert to a more standard style and provide a brief review of previously published literature on fast growing functions and Turing degrees followed by the author's own results in this area.

\subsection{Notation \& Background}

The notation we use is largely standard.   We use \( \pair{x}{y} \) to denote the integer code of the pair \( (x, y)  \), \( \TPlus_{i\in \omega} A_i \) for the set whose \( i \)-th column is \( A_i \), and \( \setcmp{C} \) to denote the compliment of \( C \).  

A string is a member of \( \wstrs \) and trees subsets of \( \wstrs \) closed under initial segments.  When we need to distinguish between strings and their integer codes we write \( \gcode{\sigma} \) for the code of  \( \sigma  \).   We use \( \sigma \incompat \tau \) and \( \sigma \compat \tau \) to denote that \( \sigma, \tau \) are incompatible and compatible respectively and write \( \sigma\concat\tau \) to denote the concatenation of the two strings.   \( \lh{\sigma} \) gives the length of \( \sigma \) and \( \strpred{\sigma} \) denotes the longest proper initial segment of \( \sigma \).  The set of (infinite) paths through a tree \( T \) is denoted \( [T] \) and \( \pruneTree{T} \) is the set of strings in \( T \) extended by some infinite path.  We call functions from \( \wstrs \) to \( \wstrs \) monotonic if it is an isomorphism of the partial ordering \( \subsetneq \) on it's domain and range.  We abuse notation and use \( T\restr{n} \) to denote the members of \( T \) of length at most \( n \) and write \( \theta(f) \) for \( \Union_{\sigma \subset f} \theta(\tau) \) when \( \theta \) is monotonic and total on \( \set{\sigma}{\sigma \subseteq f} \).

Kleene's set of ordinal notations is \( \kleeneO \), the canonical ordering of notations is \( \kleeneleq \) and \( + \) gives the effective sum of notations.  When \( \lambda \) is a limit notation we denote the \( n \)-th element of the effectively given increasing sequence defining \( \lambda \)  by \( \kleenelim{\lambda}{n} \).  We write \( \Csigman{\alpha} \) and \(  \Cpin{\alpha} \) for the collections of computably \( \sigman{\alpha}\) and \(  \pin{\alpha} \) formulas and \( \Csigman[X]{\alpha} \) and \(  \Cpin[X]{\alpha} \) when a predicate for membership in \( X \) is introduced into the language.  We use \( \LLor \) and \( \LLand \) to denote infinite disjunction  and infinite conjunction respectively.  We refer the reader to \cite{jumps-of-orderings,computable-structures-and-the-hyperarithmetical-hierarchy} for more on computable infinitary formula and to \cite{higher-recursion-theory} for more on \( \kleeneO \).

We do introduce a few non-standard pieces of notation particular to the subject matter.  Given partial functions \( r, p \) we write \( r \fung p \) if \( r(x) \geq p(x)  \) whenever they are both defined.  When \( f\) and \(  g \) are total functions we read \( f \fung g \) as \( f \) majorizes \( g \).  We say \( f \) dominates \( g \) if some \( f^{*} \) differing from \( f \) at finitely many locations majorizes \( g \).  

We indicate the local forcing relation on \( \pruneTree{T} \) by \( \frc[T] \) and it's relativization to \( \zeron{\beta} \) by \( \frc({\zeron{\beta}})[T] \) and refer the reader to \cite{jumps-of-orderings} for the definitions of the standard forcing relation \( \frc \) and \cite{classical-recursion-theory-II} for local forcing.  Informally, \( \frc[T] \) is defined in the same manner as \( \frc \) except with all quantifications over \( \wstrs \) replaced with quantification's over \( \pruneTree{T} \) (nodes in \( T \) that extend to paths).  When we extend the usual language of forcing by introducing a predicate symbol for membership in \( X \) we write \( \sigma \frc(X)[T] \phi \) to indicate that  \( \phi \) can check membership in \( X \) as an atomic operation.  Usually the set \( X \) we are forcing relative to will be clear from context and we will simply write \( \sigma \frc[T] \phi  \).  When \( g \frc(X)[T] \phi \)  or \( g \frc(X)[T] \lnot\phi \) for every \( \phi \) in \( \Csigman[X]{\beta} \) we say that \( g \) is \( \beta \) generic on \( T \) relative to \( X \).  We will take our forcing relation to denote strong forcing, that is \( \sigma  \) forces \( \phi \in \Csigman[\zeron{\beta}]{1} \) sentences only when \( \sigma \models \phi \), i.e., \( \phi \) is satisfied by referring only to information in \( \sigma \).

It is important to note that our notion of \( f \) being \( \alpha \) generic on \( T \) does not require \( f \) to force all \( \Csigman[\pruneTree{T}]{\alpha} \) facts or their negations as some definitions of genericity on a tree require \cite{reals-n-generic-relative-to-some-perfect-tree} but only \( \Csigman{\alpha} \) facts nor does it require that \( f \) be non-isolated.  Our definition is the natural way to preserve the notion of a generic path as one on which every truth is determined by a finite initial segment while requiring \( \alpha \) generic paths on \( T \) to force all \( \Csigman[\pruneTree{T}]{\alpha} \) facts or their negations extends the idea that a generic path should be typical.  Thus under our definition there is a perfect tree \( T \) with every path through \( T \) \( \alpha \) generic on \( T \) while this would be impossible under the other notion.
   
While our standard notion of forcing is concerned only with the extendable nodes on \( T \) we will also make use of a more effective notion that, by analogy with the notion of strong forcing, we call super forcing on \( T \) denoted \( \frc*[T] \).  The definition of super forcing on \( T \) exactly mirrors the definition in \cite{jumps-of-orderings} of strong forcing modified as usual to get the local forcing relation on \( T \) instead of \( \pruneTree{T} \) as above.  That is for \( \sigma \) to force \( \lnot \phi \) on \( T \) requires that every \( \tau \supset \sigma \) with \( \tau \in \pruneTree{T} \) satisfy \( \lnot \tau \frc[T] \phi \) while for \( \sigma \) to strongly force \( \lnot \phi \) on \( T \) requires this hold for every \( \tau \supset \sigma \) with \( \tau \in T \).  Hence \( \sigma \frc*[T][\zeron{\beta}] \phi \) for  \( \phi \in \Csigman[\zeron{\beta}]{\alpha} \) is \( \sigmazn[T \Tplus \zeron{\beta}]{\alpha} \) and \( \pizn[T \Tplus \zeron{\beta}]{\alpha} \) for \( \phi \in \Cpin[\zeron{\beta}]{\alpha} \).

\section{Harrington's Refutation of McLaughlin's conjecture}

In \cite{mclaughlins-conjecture} Harrington answered McLaughlin's conjecture in the negative and we will adapt his construction to establish theorem \ref{thm:nonu} but we first present his approach.  While other variations on the theme have been called McLaughlin's conjecture the form of the conjecture refuted by Harrington in \cite{mclaughlins-conjecture} is the one appearing in \cite{one-hundred-and-two-problems-in-mathematical-logic} that asserts:

\begin{conjecture}[Mclaughlin]\label{conj:mclaughlin}
Every element of a countable arithmetic subset of \( \baire \) is an arithmetic singleton.
\end{conjecture}

Harrington's refutation consisted of the following theorem.

\begin{theorem}[Harrington]\label{thm:harrington-mcl}
For every computable tree \( T_{\omega} \subset \wstrs \) is a computable tree \( T \subset \wstrs \) such \( [T] \) and \( [T_{\omega}] \) are homeomorphic and every \( f \in [T] \) is arithmetically (\( < \omega \)) generic on \( [T] \).
\end{theorem}

\begin{corollary}\label{cor:mcl-false}
 \hyperref[conj:mclaughlin]{Mclaughlin's Conjecture} is false.
\end{corollary}
We sketch how Harrington's result contradicts McLaughlin's conjecture. 

\begin{proof}
Let \( T_{\omega} \) have some non-isolated path \( f_\omega \).  Thus the homeomorphic image of \( f_\omega \), \( f \) is a non-isolated path through \( T \).  Now suppose that \( \phi(g) \) is an arithmetic predicate with unique solution \( f \).  By genericity we must have \( f \frc[T]  \phi \) hence some \( \sigma \subset f \) forces \( \phi \).  As \( f \) non-isolated there is some \( f' \neq f \) also extending \( \sigma \).  At \( f' \) is also \( < \omega \) generic on \( T \) and \( f' \frc[T]  \phi \) we have \( \phi(f') \).  Contradiction.
\end{proof}

\subsection{Sketch of the result for \( \omega \)}

Once we know that Harrington's result is true the natural approach for a recursion theorist is simply to go out and build \( T \) as some kind of distorted copy of \( T_{\omega} \) while trying to meet the genericity requirements.  The natural approach would be to simply go ahead and try to build \( T \) directly but of course if that worked straightforwardly the conjecture would likely never have remained open for as it did.  In particular the `nested' nature of the genericity requirements makes direct construction of \( T \) extremely difficult.  To force \( \Csigman{n+1} \) facts about \( f \)  we need to react to the \textit{particular} way we've failed to force  \( \Csigman{n} \) facts about \( f \).  Were we building \( f \) to be fully generic there would be no question about how \( f \) forced \( \Csigman{n} \) and \( \Cpin{n} \) facts.  We could simply read off from the definition of forcing whether a given \( \sigma \subset f \) forced some instance \( \phi(x)  \) of a \( \Cpin{n} \) formula and simply require \( f \) to extend some appropriate \( \sigma' \supset \sigma \).  But as we clearly can't build our desired \( f \) to be even fully \( 1 \)-generic here we must sometimes bring it about that \( \sigma \subset f \) forces some \( \Cpin{n} \) formula \( \lnot\phi(x) \) despite the fact that \( \phi \) is true on a co-meager set in \( \baire \) by pruning from \( T \) all extensions of \( \sigma \) that force \( \phi \).  Doing this on it's own while keeping \( T \) computable would be organizationally difficult but if we are to keep \( [T]=[T_\omega] \) we must somehow also anticipate when our commitment to somehow copy \( T_\omega \) will be incompatible with trying to force a sentence in a particular direction. Therefore, rather than a frontal assault Harrington described how we can attack the problem in reverse in a manner that provides all our organization for free.

The approach taken by Harrington rests on projecting down \( T_\omega \) to a sequence on intermediate trees \( T_n \) with \( T=T_0 \).  Each \( T_n \) will be computable in \( \zeron{n} \) and \( \pruneTree{T_n} \) will be the image of \( \pruneTree{T_{n+1}} \) under \( \theta^{n+1} \Tleq \zeron{n+1} \) where \( \theta^{n+1}  \) is monotonic and thus a homeomorphism from \( [T_{n+1}] \) to \( [T_n] \).  Furthermore if \( f \in [T_{n+1}] \) and \( f \) is \( m\) generic relative to \( \zeron{n+1}  \) on \( T_{n+1} \) then \( \theta^{n+1}(f) \) is \( m+1 \) generic relative to \( \zeron{n} \)  on \( T_n \).  Thus if this construction succeeds every \( f \in [T_0] \) is \( m \) generic on \( [T_0] \) for every \( m \in \omega \) as it's the image of some \( f_m \in T_m \) under \( \theta^{m} \compose \theta^{m-1} \compose \ldots \theta^1 \) which we abbreviate as \( \treemap{m}{0}  \).

To ensure that \( T_0 \) is homeomorphic with \( T_{\omega} \) Harrington also required that \( T_{n}\restr{m} = T_{\omega}\restr{m} \) and that \( \lh{\theta^n(\sigma)} \geq \lh{\sigma} \).  Now define \( \treemap{\omega}{0}(\sigma) \) to be \( \treemap{\lh{\sigma}}{0}(\sigma) \) and note that \( \treemap{\omega}{0}(\sigma) \in T_0 \iff \sigma \in T_m \iff \sigma \in T_\omega \).  Hence \( \treemap{\omega}{0}(\sigma) \) is clearly a continuous bijection between \( T_\omega \) and \( T_0=T \).  The only remaining problem is to construct such a sequence.  The trick here is to observe that \( T_0\restr{l} \) depends (more or less) only on \( T_1\restr{l}\) and \( T_1\restr{l}\) depends only on \( T_2\restr{l}\) and so on.  Thus \( T_0 \) can be built by looking only at \( T_l\restr{l}=T_\omega\restr{l} \).  This argument isn't too difficult to formalize in terms of the recursion theorem but unfortunately many of the important applications, including the ones we use later in this paper, depend on proving the result with an arbitrary computable ordinal \( \alpha \) substituted for \( \omega \) so we must give the fully general construction.  As we will see that while conceptually identical the technical details Harrington avoided spelling out are definitely not trivial.

\section{Harrington's Result}\label{sec:har-result}

\begin{theorem}[Harrington]\label{thm:harrington-mcl-ordinal}
For every ordinal notation \( \alpha \)  and tree \( T \Tleq \zeron{\alpha} \) there is a computable tree \( T_0 \) such \( [T_0] \) and \( [T] \) are homeomorphic and every \( f \in [T_0] \) is \(  \alpha \) generic on \( [T] \).
\end{theorem}

\subsection{Preliminaries}\label{subsec:prelim}

While this tells us how to copy \( T_{\lambda}\restr{l} \) down to lower trees it's no longer obvious how much we should copy.  To extend Harrington's argument to a sequence of length \( \alpha \) we will need to somehow specify an integer \( \copylen{\beta} \) for every \( \beta \kleenel \alpha \) telling us how much of \( T_{\beta+1} \) we should copy down to \( T_{\beta} \).   We need to ensure that \( T_{\lambda} \) will be the limit of \( T_\beta \) for \( \beta \kleenel \lambda \) so \( [T_0]=[T_\alpha] \) and that we can define \( T_{\kleenelim{\lambda}{n}}\restr{\copylen{\kleenelim{\lambda}{n}}} \) from only \( T_{\kleenelim{\lambda}{n}} \) to ensure our fixed point is non-empty so we must have:

\begin{align}
		& \forall[\gamma]\left( \kleenelim{\lambda}{n} \kleeneleq \beta \kleenel \lambda \implies \copylen{\kleenelim{\lambda}{n}} \leq \copylen{\beta}  \right)\label{edef:copylen-between} \\
	 & \lim_{n \to \infty} \copylen{\kleenelim{\lambda}{n}} = \infty \\
	\shortintertext{For concreteness we will also insist that}
	&\label{edef:copylen} \copylen{\beta} = \begin{cases}
	 												0 & \text{if } \enumpred{\beta}=\diverge\\
													\copylen{\enumpred{\beta}}+n & \text{if }  \beta=\kleenelim{\enumpred{\beta}}{n} 
											\end{cases}
\end{align}

The difficulty in achieving these conditions is that in general a notation \( \beta \) could appear at arbitrary places in the effective limit for arbitrarily many \( \lambda \kleeneg \beta \).  A further difficulty is posed by the need to build a single function \( \copylen{\beta} \) defined on a path through \( \kleeneO \) as required by some of the corollaries.   Our strategy is to associate to each \(\beta \) a unique limit notation \( \enumpred{\beta} \) to whose effective limit \( \beta \) belongs.  Since this proof is fairly technical we delay it's presentation until the appendix and blithely continue assuming we have a computable function \( \copylen{\beta} \) satisfying the above (below \( \alpha \)) and that every \( \beta \) appears in at most one effective limit (below \( \alpha \)) denoted \( \enumpred{\beta} \) (extended to be total, increasing, limit valued).  This is slightly inaccurate, but we reserve those qualifications for the appendix.

This resolves the problem of how much to copy but we don't yet know exactly what to copy.  In the sketch of theorem \ref{thm:harrington-mcl} we had a computable tree \( T_\omega \) but in general at limit stages \( T_{\lambda} \) will only be computable in \( \zeron{\lambda} \) but we will still need to copy \( T_\lambda\restr{\copylen{\kleenelim{\lambda}{n}}} \) down to \( T_{\kleenelim{\lambda}{n}} \Tleq \zeron{\kleenelim{\lambda}{n}} \).  We show that we can always convert our trees to a form in which the segments requiring copying can always be uniformly recovered from the appropriate degree.

\begin{definition}\label{def:uniform-tree}
	Say the tree \( T_\beta \Tleq \zeron{\beta} \) is uniform if \( T_\beta\restr{\copylen{\kleenelim{\beta}{n}}} \) is uniformly computable from \( \zeron{\kleenelim{\beta}{n}} \) for all \( n \).
\end{definition}

Converting \( T_\lambda \) into a uniform tree simply requires we delay killing branches in \( T_\lambda \) until they are long enough that we can use enough of \( \zeron{\lambda} \) to verify the branch gets killed.

\begin{lemma}\label{lem:convert-to-n-bit}
Suppose \( T_\lambda \Tleq \zeron{\lambda} \), \( T_\lambda\restr{\copylen{\lambda}}=T_{\enumpred{\lambda}}\restr{\copylen{\lambda}} \) and \( T_{\enumpred{\lambda}} \) is uniform then there is a uniform \( \hat{T_\lambda} \) with \( [T_\lambda]=[\hat{T_\lambda}] \) and \(  T_\lambda\restr{\copylen{\lambda}}=\hat{T}_\lambda\restr{\copylen{\lambda}} \).  Furthermore, this holds with all possible uniformity.  
\end{lemma}
 \begin{proof}
	As \( T_{\enumpred{\lambda}} \) is uniform we may let \( \hat{T}_\lambda\restr{\copylen{\lambda}}=T_\lambda\restr{\copylen{\lambda}} \) without difficulty.  Now given \( \sigma \in \bstrs \) with \( \copylen{\lambda} < \lh{\sigma} =l \) place \( \sigma \in \hat{T_\lambda} \) unless some computation showing that \( \sigma \nin T_\lambda \) converges in at most \( l \) many steps while consulting only those columns of \( \zeron{\lambda} \) that encode  \( \zeron{\gamma} \) for \( \gamma \kleeneleq\kleenelim{\lambda}{l} \).   The uniformity is evident in the proof.
\end{proof}

For the remainder of the paper we will apply the preceding lemma without comment and assume without comment that any needed conversion of this kind is done behind the scenes.

\subsection{The Desired Sequence}

In this section we fix some ordinal notation \( \alpha \) and work to build a tower of trees \( \seq{T_{\beta}}{\beta \kleeneleq \alpha} \) as we sketched for \( \omega \) (technically speaking \( \alpha \) isn't truly arbitrary only the ordinal it denotes is).  We first describe the properties of the trees we seek to build.  At each \( \beta \) we will try to build our tree \( T_\beta \Tleq \zeron{\beta} \) so that every path meets every \( \sigmazi[\zeron{\beta}] \) set as soon as possible.  We capture the effect of this construction below with the notion of eagerly generic (meaning \( 1 \) generic over \( \zeron{\beta} \)). 

\begin{definition}\label{def:eager}
	Say that \( T \Tleq \zeron{\beta} \) is eagerly generic (over \( \zeron{\beta} \)) if for all \( f \in [T] \) and \( \psi \in \Csigman[\zeron{\beta}]{1} \) there is some \( \sigma \subset f \) such that either \( \sigma \) super forces \( \psi \) (\( \sigma \) witnesses the \( \Csigman[\zeron{\beta}]{1} \) fact) or \( \sigma \) super forces \( \lnot\psi \) (no \( \tau \in T \) (not \( \pruneTree{T} \)) super forces \( \psi \)).
\end{definition}

Note that this immediately entails that every  \( f \in [T_\beta] \) is \( 1 \) generic relative to \( \zeron{\beta} \) on \( T_\beta \) and the relation \( f \frc({\zeron{\beta}})[T_\beta] \psi \)  is computable in \( \zeron{\beta+1} \) for \( \psi \in \Csigman[\zeron{\beta}]{1} \union \Cpin[\zeron{\beta}]{1} \).  Since forcing on \( T_\beta \) always occurs relative to \( \zeron{\beta} \) we will abbreviate \( \frc({\zeron{\beta}})[T_\beta] \) as \( \frc[T_\beta] \) or even \( \frc[\beta] \) where this won't generate confusion. We are now ready to precisely state the conditions our trees aim to satisfy.

\begin{definition}\label{def:downwardly-generic}
	Say that a sequence \( \seq{T_\beta}{\beta \kleeneleq \alpha} \) is a downwardly generic tower (of length \( \alpha \)) if
	\begin{arabiclist}
		\item For \( \beta \kleeneleq \alpha \),  \( T_\beta \Tleq \zeron{\beta} \)\label{def:downwardly-generic:beta-computable}.
		\item For \( \beta \kleeneleq \alpha \), \( T_\beta \) is  eagerly generic over \( \zeron{\beta} \)\label{def:downwardly-generic:eagerly}.
		\item If \( \beta +1 \kleeneleq \alpha \) there is a monotonic embedding \(  \theta^{\beta +1} \Tleq \zeron{\beta+1} \) of  \( T_{\beta+1} \) into \( T_\beta \) that induces a homeomorphism of \( [T_{\beta+1}] \) and \( [T_\beta] \).\label{def:downwardly-generic:map}
		\item \ref{def:downwardly-generic:map} and \ref{def:downwardly-generic:beta-computable} hold uniformly in \( \beta \).\label{def:downwardly-generic:unifom}
		\item \( T_\beta\restr{\copylen{\beta}}=T_{\beta+1}\restr{\copylen{\beta}} \).\label{def:downwardly-generic:copying}
		\item If \( \lambda \) a limit \( T_\lambda = \lim_{\beta \kleenel \lambda} T_\beta \)\label{def:downwardly-generic:lim-unions}
	\end{arabiclist}
\end{definition}

Note that by \eqref{edef:copylen} these last two conditions entail that \( T_\lambda\restr{\copylen{\lambda}}=T_{\enumpred{\lambda}}\restr{\copylen{\lambda}} \).  For the remainder of this subsection we fix a downwardly generic tower \( \seq{T_\beta}{\beta \kleeneleq \alpha} \) of length \( \alpha \) and proceed to demonstrate it has the desired properties.  First we observe that our tower preserves the topological structure of \( T_{\alpha} \) and provides a moderately effective means of translation between trees at different levels.

\begin{lemma}\label{lem:tower-homeomorphism}
	For every \( \gamma \kleenel \beta \kleeneleq \alpha \) there is a monotonic function  \( \treemap{\beta}{\gamma}\map{T_\beta}{T_\gamma} \) uniformly computable in \( \zeron{\beta} \) embedding \( T_\beta \) into \( T_\gamma \) so as to induce a homeomorphism of \( [T_\beta] \) and \( [T_\gamma] \) uniquely defined by the constraints:
	\begin{arabiclist}
		\item If \( \beta \kleenel \gamma \kleenel \lambda \) then \( \treemap{\lambda}{\beta}=\treemap{\gamma}{\beta} \compfunc\treemap{\lambda}{\gamma} \)
		\item \( \treemap{\beta+1}{\beta}=\theta^{\beta+1} \)\label{lem:tower-homeomorphism:extend-theta}\\
		\item If \( \lh{\sigma} \leq \copylen{\beta} \) then \( \treemap{\enumpred{\beta}}{\beta}(\sigma) = \sigma \) \label{lem:tower-homeomorphism:limits}.
	\end{arabiclist}
\end{lemma}

\begin{proof}
We use the method of effective transfinite recursion by assuming we have some index \( e \) such that \( \recfnl{e}{\zeron{\gamma}}{\gamma,\beta,\sigma}=\treemap{\gamma}{\beta}(\sigma) \) for every \( \gamma \kleenel \kappa \) and build a computable function \( I(e) \) so that \( \recfnl{I(e)}{\zeron{\gamma}}{\gamma,\beta,\sigma}=\treemap{\gamma}{\beta}(\sigma) \) for every \( \gamma \kleeneleq \kappa \) and then use the fixed point lemma to build a single computable function working for all \( \gamma \kleeneleq \alpha \).    The behavior of \( I(e) \) is spelled out plainly for \( \kappa \) a successor (making use of part \ref{def:downwardly-generic:unifom} to recover the various indexes) and for \( \kappa \) a limit  \( \recfnl{I(e)}{\zeron{\kappa}}{\kappa,\beta,\sigma} \) computes the the value \( \recfnl{e}{\zeron{\kleenelim{\kappa}{l}}}{\kleenelim{\kappa}{l},\beta,\sigma} \) where \( l > \lh{\sigma} \) is the first such integer for which we observe \( \kleenelim{\kappa}{l} \kleeneg \beta \).  As \( \copylen{\kleenelim{\kappa}{l}} \geq l \) the fixed point of \( I(e) \) plainly defines a monotonic function for all \( \beta \kleenel \gamma \kleeneleq \alpha \) mapping \( T_\gamma \) into \( T_\beta \).  By \eqref{edef:copylen-between} and part \ref{def:downwardly-generic:copying} of definition \ref{def:downwardly-generic} any valid choice of \( l \) yields the same result establishing uniqueness.   

This leaves only the claim that \( \treemap{\beta}{\gamma} \) gives a surjection of \( \pruneTree{T_{\gamma}} \) onto  \( \pruneTree{T_{\beta}} \) to verify. Assume not and let \( \gamma, \beta \) be the lexicographically least such that the claim fails for \( \treemap{\gamma}{\beta} \).  Now clearly, by the minimality of \( \gamma \) and property \ref{def:downwardly-generic:map} of definition \ref{def:downwardly-generic}, \( \gamma\) evidently can't be a successor.  But if \( \gamma \) a limit, \( g \in [T_\beta] \) and \(  \kleenelim{\gamma}{l} \kleeneg \beta \) then by the minimality of \( \gamma \) there is some some \( \sigma_l \in T_{\kleenelim{\gamma}{l}} , \lh{\sigma_l} = l \) with  \( \treemap{\kleenelim{\gamma}{l}}{\beta}(\sigma_l) \supseteq g\restr{l} \).  By condition \ref{lem:tower-homeomorphism:limits} \( \treemap{\kleenelim{\gamma}{l+1}}{\kleenelim{\gamma}{l}}(\sigma_{l+1}) = \sigma_{l+1} \) for all large enough \( l \) so by monotonicity \( \sigma_{l+1} \supseteq \sigma_l \) and by part \ref{def:downwardly-generic:copying} of definition \ref{def:downwardly-generic}  \( \sigma_l \in T_{\gamma} \).  Thus \( h=\Union_{l \in \omega} \sigma_l \) is a path through \( T_\gamma \) with \( \treemap{\gamma}{\beta}(h)=g \).
\end{proof}

With this lemma in mind we adopt the notation that if \( g \) is a path through \( T_0 \) then \( g_\beta \) refers to the path in \( T_\beta \) which maps to \( g \) under \( \treemap{\beta}{0} \) which we abbreviate \( \treemap{\beta}{} \).  We now work toward showing \( T_0 \) will be sufficiently generic by translating \( \Csigman{\beta+1} \), \( \Cpin{\beta+1}  \) formulas about \( g \in T_0 \) to  \( \Csigman[\zeron{\beta}]{1} \) formulas about \( g_\beta \) so that if \( g_\beta \) forces the translated formula on \( T_{\beta} \) then \( g \) forces the original on \( T_0 \).  Our strategy will be to eliminate quantifiers from the interior of a formula  by replacing \( \psi \in \Csigman[\zeron{\gamma}]{1} \union \Cpin[\zeron{\gamma}]{1}  \) with the \( \Csigman[\zeron{\beta}]{1} \) formula asserting that \( \treemap{\beta}{\gamma}(g_{\beta})=g_\gamma \) forces \( \psi \) when \( \beta \kleeneg \gamma \).

\begin{definition}\label{def:psi-star}
	Given a computable infinitary formula \( \phi \) we define \( \phi^{\beta} \) inductively as follows:

	\begin{align}
		\phi_0 &= \phi \\
		\phi^{\beta+1} &= \begin{cases}
													\exists({\sigma \subset g_{\beta+1}})\left(\theta^{\beta+1}(\sigma) \frc[\beta] \phi^{\beta}\right) & \text{if } \phi^\beta \in \Csigman[\zeron{\beta}]{1} \cup \Cpin[\zeron{\beta}]{1} \\
													\LLor_{i \in \omega} \psi^{\beta+1}_{q(i)} & \text{ otherwise when } \phi=\LLor_{i \in \omega} \psi_{q(i)} \\
													\lnot \psi^{\beta+1} & \text{ otherwise when } \phi=\lnot \psi \\
										\end{cases} \\
	 \shortintertext{For \( \lambda \) a limit}
		\phi^{\lambda} &= \begin{cases}
													\exists({\sigma \subset g_\lambda})\left(\treemap{\lambda}{\gamma}(\sigma) \frc[{\gamma}]  \phi^\gamma\right) &\text{if } \phi \in \Csigman{\gamma} \union \Cpin{\gamma} \text{ for } \gamma \kleenel \lambda \\
													\LLor_{i \in \omega} \psi^{\lambda}_{q(i)} &\text{otherwise when } \phi=\LLor_{i \in \omega} \psi_{q(i)} \\
													\lnot \psi^{\lambda} &\text{otherwise when } \phi=\lnot \psi
										\end{cases}
	\end{align}
\end{definition}

\begin{lemma}\label{lem:psi-star-forcing}
	If \( \sigma \in \pruneTree{T_\beta} \) then \( \sigma \frc[\beta] \phi^{\beta} \)  iff \( \treemap{\beta}{}(\sigma) \frc[0] \phi \).  Moreover, if \( \phi \in \Csigman{\beta+1} \union \Cpin{\beta+1} \)  then \( \phi^{\beta} \in  \Csigman[\zeron{\beta}]{1}   \)
\end{lemma}

\begin{proof}
Fix \( \gamma \) to be least ordinal for which the equivalence \( \sigma \frc[\gamma] \phi^{\gamma} \iff \treemap{\gamma}{}(\sigma) \frc[0] \phi \) fails  to hold for some \( \phi \) and let \( \phi \) be the witness to this failure of least complexity.  Clearly \( \gamma \neq 0 \) so first suppose \( \gamma=\beta+1 \).
	
First suppose \( \phi^{\beta} \in \Csigman[\zeron{\beta}]{1} \cup \Cpin[\zeron{\beta}]{1} \).  In this case \( \phi^{\beta+1} \) is defined as in the first case above so if \( \sigma \frc[{\beta+1}] \phi^{\beta+1} \) then there must actually be some \( \sigma' \subseteq \sigma \) with \( \theta^{\beta+1}(\sigma') \frc[\beta] \phi^{\beta} \).  By monotonicity \( \theta^{\beta+1}(\sigma) \frc[\beta] \phi^{\beta} \) and  by the inductive hypothesis \( \treemap{\beta}{0}(\theta^{\beta+1}(\sigma)) = \treemap{\beta+1}{}(\sigma) \frc[0] \phi \).  Alternatively suppose that \( \phi^{\beta} \nin \Csigman[\zeron{\beta}]{1} \cup \Cpin[\zeron{\beta}]{1} \).  But if \( \sigma \frc[{\beta+1}] \LLor_{i \in \omega} \psi^{\beta+1}_{q(i)} \) then for some \( i \) we must have \( \sigma \frc[{\beta+1}] \psi^{\beta+1}_{q(i)} \) so by the minimality of \( \phi \) we have \( \treemap{\beta+1}{0}(\sigma) \frc[0] \psi^{\beta+1}_{q(i)} \) and therefore \( \treemap{\beta+1}{}(\sigma) \frc[0] \phi^{\beta+1}  \).   Finally if \(\phi= \lnot \psi \) and \( \sigma \frc[{\beta+1}] \lnot \psi^{\beta+1} \)  then again by the minimality of \( \phi \) no \( \tau \supseteq \treemap{\beta}{}(\sigma) \) in \( \pruneTree{T_0} \) forces \( \psi \) so therefore \( \treemap{\beta}{}(\sigma) \frc[T_0] \lnot \psi=\phi \).  Consequently \( \rightarrow \) can't fail at \( \gamma=\beta+1 \).
	
Going the other way if we suppose that  \( \phi^{\beta} \in \Csigman[\zeron{\beta}]{1} \cup \Cpin[\zeron{\beta}]{1} \) and that \( \treemap{\beta+1}{}(\sigma)=\treemap{\beta}{0}(\theta^{\beta+1}(\sigma)) \frc[0] \phi \) then minimality ensures that \( \theta^{\beta+1}(\sigma) \frc[\beta] \phi^{\beta} \) so by definition \(\sigma \models \phi^{\beta+1}\) holds entailing \( \sigma \frc[{\beta+1}] \phi^{\beta+1}  \).  Alternatively, suppose that \( \treemap{\beta+1}{}(\sigma) \frc[T_0] \phi= \LLor_{i \in \omega} \psi_{q(i)} \) then for some \( i \) we have \( \treemap{\beta+1}{}(\sigma) \frc[T_0] \psi_{q(i)} \) so by minimality of \( \phi \) we infer \( \sigma \frc[{\beta+1}] \psi^{\beta+1}_{q(i)}  \) and therefore \( \sigma \frc[{\beta+1}] \phi^{\beta+1}  \).  Lastly if \( \treemap{\beta+1}{0}(\sigma) \frc[0] \phi= \lnot \psi \) then by other direction no extensions of \( \sigma \) on \( \pruneTree{T_{\beta+1}} \) can force \( \psi^{\beta+1} \) or some extension of \(  \treemap{\beta+1}{0}(\sigma) \) would force \( \psi \) so \( \sigma \) must force \( \phi^{\beta+1} = \lnot \psi^{\beta+1}\). 
	
The proof for limit stages follows by the same considerations and the last claim follows by straightforward induction.
\end{proof}

\begin{lemma}\label{lem:totally-generic}
Every \( g \in [T_0] \) is \(  \alpha \)-generic on \( T_0 \).
\end{lemma}

\begin{proof}
Fix \( \phi \in \Csigman{\alpha} \).  By lemma \ref{lem:psi-star-forcing} \( \phi^{\alpha} \in \Csigman[\zeron{\alpha}]{1} \) so either \( g_\alpha \) forces \( \phi^{\alpha} \) or \( \lnot \phi^{\alpha} \) as every path through \( T_\alpha \) is eagerly generic so applying lemma \ref{lem:psi-star-forcing} again we conclude that \( g  \) forces either \( \phi \) or it's negation on \( T_0 \). 
 \end{proof}

\subsection{The Construction}

We now demonstrate the existence of a downwardly generic tower of length \( \alpha \).  Our construction will begin with an arbitrary tree \( \hat{T_\alpha} \) computable in \( \zeron{\alpha} \) which we will modify to be an eagerly generic tree \( T_{\alpha} \).  From \( T_{\alpha} \) we will work downward to define \( T_\beta \) for \( \beta \kleenel \alpha \) by way of the following effective process. 

\begin{lemma}\label{lem:tree-minus-one}
	Given trees \( T_{\beta+1} \Tleq \zeron{\beta+1} \) and \( T_{\gamma} \Tleq \zeron{\gamma} \) with \( \enumpred{\beta}=\gamma \) there is a tree \( T_\beta \Tleq \zeron{\beta} \) with and a monotonic function \( \theta^{\beta+1} \) such that
	\begin{arabiclist}
		\item \( T_{\beta}\restr{\copylen{\beta}}=T_{\gamma}\restr{\copylen{\beta}} \).\label{lem:tree-minus-one:copy}
		\item If \( T_{\beta+1}\restr{\copylen{\beta}}=T_{\gamma}\restr{\copylen{\beta}} \) then \( \theta^{\beta+1} \) is a monotonic embedding of \( T_{\beta+1} \) into \( T_\beta \) such that \(  [\theta^{\beta+1}(T_{\beta+1})]= [T_\beta]  \)\label{lem:tree-minus-one:map}
		\item \( T_\beta \) is eagerly generic.\label{lem:tree-minus-one:eagerly}
\end{arabiclist}
		Furthermore indexes for \( T_\beta, \theta^{\beta+1} \) are computable from the indexes for \( T_{\beta+1} \) and \( T_\gamma \) via a function that is total even when passed indexes for \( T_{\beta+1} \) and \( T_\gamma \) that fail to converge on some values.
\end{lemma}

\begin{proof}
For simplicity we refer to \( T_{\beta+1} \) as \( \hat{T} \), \( \theta^{\beta+1} \) as \( \theta \) and \( T_\beta \) as \( T \).  We fix a \( \zeron{\beta} \) stagewise approximation to \( \hat{T} \) valid in the limit and (implicitly using lemma \ref{lem:convert-to-n-bit}) we set \( T\restr{\copylen{\beta}}=T_\gamma\restr{\copylen{\beta}} \).  We set \( \theta \) to be the identity on \( T\restr{\copylen{\beta}} \) and pause the entire construction at any stage where the approximation to \( \hat{T} \) doesn't agree with \( T\restr{\copylen{\beta}} \).  Thus should the condition \( \hat{T}\restr{\copylen{\beta}}=T_\gamma\restr{\copylen{\beta}} \) fail, our construction eventually shuts down and refuses to produce a useful result.  Thus we've directly satisfied part \ref{lem:tree-minus-one:copy} of the lemma.
	
For \( \sigma \) extending some element in \( T\restr{\copylen{\beta}} \) we define \( \theta(\sigma) \) to be the limit as \( s \) goes to infinity of \( \theta_s(\sigma) \). To ensure \( T \) is computable from \( \zeron{\beta} \) we decide whether \( \sigma \) is in \( T \) at the first stage \( s \) greater than the code of \( \sigma \)  by placing it in \( T \) if it is in the range of \( \theta_s \).  At all times we maintain that if \( \theta_s(\sigma) \) is defined and \( \tau \subseteq \sigma \) then  \( \tau \in \hat{T}_s \) by letting \( \theta_s(\sigma) \) become undefined if when required.  If at stage \( s \) we observe some \( \sigma=\tau\concat[k] \) with code at most \( s \)  to be in \( \hat{T}_s \) and \( \theta_s(\tau) \) is defined but \( \theta_s(\sigma) \) undefined we then also define \( \theta_s(\sigma)=\theta_s(\tau)\concat[\pair{k}{s}] \).

We guarantee that \ref{lem:tree-minus-one:eagerly} of the lemma holds by ensuring that if \( \lh{\sigma}=2i > \copylen{\beta} \) then either \( \theta(\sigma) \) meets \( \REset({\zeron{\beta}}){i} \) or for all \( \tau \in T\) with \( \tau \supset \theta(\sigma) \) \( \tau \) does not meet \( \REset({\zeron{\beta}}){i} \). This is accomplished simply by letting \( \theta_{s+1}(\sigma) \) be redefined to equal any \( \tau \supset \theta(\sigma) \) with \( \tau \in \REset({\zeron{\beta}})[s]{i} \isect T_s \) and reseting all \( \theta_{s+1}(\sigma') \) for \( \sigma' \supsetneq \sigma \) to undefined. Thinking of this as a finite injury argument we note that if \( \sigma \in \hat{T} \) eventually we reach some stage \( t \) so that at any later stage \( s \), \( \sigma \) and all of it's initial segments are members of \( \hat{T}_s \). Furthermore if \( \lh{\sigma} \leq 2k \) then we redefine \( \theta_s(\sigma) \) no more than \( 2^k \) times after stage \( t \) in attempts to meet \ce in \( \zeron{\beta} \) sets. It is therefore clear that eventually \( \theta_s(\sigma) \) will stabilize. Moreover, for the set of extensions of \( \sigma \) in \( T \) to be infinite there must be infinitely many stages \( s \) in which \( \sigma \) was in the range of \( \theta_s \) so if \( \sigma \subset g \in [T] \) then \( \sigma \) is in the range of \( \theta \) thus part \ref{lem:tree-minus-one:map} of the lemma is satisfied.

The uniformity is evident from the proof, but some remarks about why the resulting function is total even when passed partial indexes is warranted.  With respect to  \( T_{\beta+1} \) all that is really necessary is that eventually all members of \( T_{\beta+1} \) stay in the approximation while non-members are out of the approximation at infinitely many stages so being \ce in \( \zeron{\beta+1} \) would suffice.  Since lemma \ref{lem:convert-to-n-bit} only cared about elements being enumerated into the compliment of \( T_{\gamma} \) that index may also be partial.
\end{proof}

Note that the condition \( T_{\beta+1}\restr{\copylen{\beta}}=T_{\gamma}\restr{\copylen{\beta}} \) in the above lemma is guaranteed to be satisfied if \( T_{\beta+1} \) properly copies \( T_{\enumpred{(\beta+1)}} \) by lemma \ref{lem:build-copylen}.  Also remember that should \( \enumpred{\beta} \kleeneg \alpha \) we defined \( T_{\enumpred{\beta}} \) to be another copy of \( T_\alpha \).  Since we only make use of \( T_{\enumpred{\beta}} \) to copy \( T_{\enumpred{\beta}}\restr{\copylen{\beta}} \) we may safely pretend (by redefinition) that \( \enumpred{\beta} = \alpha \) whenever it would otherwise be larger than \( \alpha \).

\begin{lemma}\label{lem:downward-tower-exists}
Given \( T \Tleq \zeron{\alpha} \) there is a downwardly generic tower \( \seq{T_\beta}{\beta \kleeneleq \alpha} \) of length \( \alpha \) with \( T_\alpha\) the effectively given image of \( T \).
\end{lemma}

\begin{proof}
By the same argument given in lemma \ref{lem:tree-minus-one} we can effectively transform \( T \) into an eagerly generic \( T_\alpha \Tleq \zeron{\alpha} \).  Furthermore we may assume that \( \alpha \) is a limit ordinal during construction by applying lemma  \ref{lem:tree-minus-one} finitely many times until we reached a limit level.   Now fix an index \( i \) for \( T_\alpha \) and let \( J(\beta,j,j') \) be the computable function giving an index for \( T_\beta \) given an index \( j \) for \( T_{\beta+1} \) and \( j' \) for \( T_{\enumpred{\beta}} \).  We now define a computable function \( I(e) \) to behave as follows with the intent that \( \recfnl{I(e)}{}{} \) should define a function from notations \( \beta \kleeneleq \gamma \) to an index for \( T_\beta \) whenever \( \recfnl{e}{}{} \) defines the same function on \( \beta \kleenel \gamma \).
	\begin{equation}
		\recfnl{I(e)}{}{\beta} = \begin{cases}
																\diverge & \text{ unless } \beta \kleeneleq \alpha \\
																 i & \text{ if } \beta = \alpha\\
																 J(\beta,\recfnl{e}{}{\beta+1},\recfnl{e}{}{\enumpred{\beta}} ) & \text{ otherwise}
		\end{cases}
	\end{equation}

Fix \( e \) to be a fixed point of \( I(e) \) and let \( T_\beta \) be the tree defined by index \( \recfnl{e}{}{\beta} \) relative to \( \zeron{\beta} \) for \( \beta \kleenel \alpha \).  Note that it is enough to show that \( T_\beta \) is built as per lemma \ref{lem:tree-minus-one} from \( T_{\beta+1} \) and \( T_{\enumpred{\beta}} \) since lemma \ref{lem:build-copylen} ensures that if \( T_{\enumpred{\beta}}\restr{\copylen{\beta}}=T_\beta\restr{\copylen{\beta}} \) then \( T_{\beta+1}\restr{\copylen{\beta}}=T_\beta\restr{\copylen{\beta}} \) and as well as that \( T_\lambda \) for \( \lambda \) a limit is the limit of \( T_\beta \) with \( \beta \kleenel \lambda \)  \ref{def:downwardly-generic}.  
	
Now suppose that \( T_\beta \) fails to be defined or satisfy the conditions of lemma \ref{lem:tree-minus-one} with respect to \( T_{\beta+1} \) and \( T_{\enumpred{\beta}} \).  Since there are no infinite decreasing sequences of ordinals we can assume that \( T_{\enumpred{\beta}} = \REset(\zeron{\enumpred{\beta}}){q} \) where \( q = \recfnl{e}{}{\enumpred{beta}} \) is defined and satisfies the conclusions of lemma \ref{lem:tree-minus-one}.  Thus if \( r=\recfnl{e}{}{\beta} \)  by the choice of \( e \) as a fixed point we also have \( r=J(\beta,\recfnl{e}{}{\beta+1},\recfnl{e}{}{\enumpred{\beta}} \) so \( T_\beta \) is defined by the application of lemma \ref{lem:tree-minus-one}.  Note that the work here is really being done by lemma \ref{lem:tower-homeomorphism} which verified that merely being the image of a monotonic function and the properties of the function \( \copylen{\beta} \) ensure that all trees in the tower are homeomorphic.
\end{proof}

This completes the proof of theorem \ref{thm:harrington-mcl-ordinal}.  At this point it is interesting to note that this is in some sense optimal since every member of a countable hyperarithmetic class \( A \subset \baire \) is itself a hyperarithmetic singleton.

\begin{definition}\label{def:alpha-root}
	Say \( T \) is the \( \alpha \)-reduct of \( \hat{T} \) if \( T=T_0 \) where \( T_0 \) is constructed as described above from  \( \hat{T} \Tleq \zeron{\alpha} \).  Also we call those  \( g \in [T_0] \) an \( \alpha \) root of \( \hat{g} \in [\hat{T}] \) if \( g \) is the image of \( \hat{g} \) under the constructed homomorphism.
\end{definition}

Note that an index for the \( \alpha \)-reduct of \( \hat{T} \) as a computable set can be effectively computed from a  index for \( \hat{T}  \) as a \(  \zeron{\alpha} \) computable set.

\subsection{Consequences}\label{sec:consequences}

Harrington observed several other important consequences of the above method in \cite{mclaughlins-conjecture} that have also never been formally published and we take the time to present those that can be stated in terms of classical computability theory here and leave those about admissible sets and various implications in second order arithmetic to another paper.

\begin{definition}\label{def:subgeneric}
	Following Harrington \cite{mclaughlins-conjecture} we say a degree \( \Tdeg{d} \in \baire \)  is \( \alpha \) subgeneric for \( \alpha \in \kleeneO \) if for all \( \beta \kleenel \alpha \) \( \Tdeg{d} \) satisfies both
	\begin{arabiclist}
		\item \(\displaystyle  \jumpn{\Tdeg{d}}{\beta} \Tequiv \Tdeg{d} \Tjoin \zeron{\beta} \)\label{def:subgeneric:join-jump}
		\item \(\displaystyle  \forall[X]\left( X \Tleq \zeron{\alpha} \land X \Tleq \jumpn{\Tdeg{d}}{\beta} \implies X \Tleq \zeron{\beta} \right) \)\label{def:subgeneric:meet-zeron-beta}
	\end{arabiclist}
\end{definition}

A version of Harrington's first corollary in  \cite{mclaughlins-conjecture}  can now be stated.

\begin{corollary}[Harrington \cite{mclaughlins-conjecture}]\label{cor:subgeneric}
	For each \( \alpha < \wck \) there is a sequence \( \seq{g_n \in \baire}{n \in \omega} \) so that for all \( n \in \omega \)
	\begin{arabiclist}
		\item \( \displaystyle \Tdeg{g_n}  \) is \( \alpha \) subgeneric.
		\item \(\displaystyle  g_n \nTleq \left(\TPlus_{i \neq n} \jumpn{g_i}{\alpha} \right) \)\label{cor:subgeneric:alpha-independence}.
		\item \( g_n \) is a the unique solution of a \( \pizn{1} \) formula the index for which is given uniformly in \( n \).
	\end{arabiclist}
\end{corollary}

Our first task is to assure ourselves we already know how to satisfy part \ref{def:subgeneric:join-jump} of definition \ref{def:subgeneric}.  

\begin{lemma}\label{lem:alpha-subgeneric-part-1}
If \( g \) is an \( \alpha \) root then \( g \) satisfies part \ref{def:subgeneric:join-jump} of the \hyperref[def:subgeneric]{definition of \( \alpha \)-subgeneric}.
\end{lemma}

\begin{proof}
Fix \( \seq{T_\beta}{\beta \kleeneleq \alpha} \) witnessing that \( g \) is an \( \alpha \) root and  \( \lambda \kleeneleq \alpha \) and \( \psi(x) \in \Csigman{\lambda} \) such that \( x \in \jumpn{g}{\lambda} \iff g \models \psi \).   By lemma \ref{lem:psi-star-forcing}  if \( g \in [T_0] \) then \( g \models \psi \iff g_{\lambda} \frc[\lambda] \psi^\lambda \).   Since either \( \psi^\lambda \) or it's negation is super forced on \( T_\lambda \Tleq \zeron{\lambda}  \) by \( g^\lambda \Tleq \zeron{\lambda} \Tjoin g \) we can compute \( \jumpn{g}{\lambda} \) from \( \zeron{\lambda} \Tjoin g \) completing the proof.
\end{proof}

Building \( g \) as an \( \alpha \) root that also satisfies part \ref{def:subgeneric:meet-zeron-beta} of the \hyperref[def:subgeneric]{definition of \( \alpha \)-subgeneric} requires slightly more work.  Given \( X \Tleq \zeron{\beta+1} \) and \( X \Tleq \zeron{\beta} \Tjoin g \) these computations must be super forced on \( T_{\beta+1} \) to be equal but we need to guarantee they are super forced to agree on \( T_\beta \) to ensure \( X \Tleq \zeron{\beta} \).  Since \( T_\beta \) isn't the image of \( T_{\beta+1} \) under \( \theta^{\beta+1} \) super forcing on \( T_{\beta+1} \) doesn't translate to super forcing on \( T_\beta \) so we must guarantee this occurs manually.  Since \( T_\beta \) lacks access to \( \zeron{\beta+1} \) we can't directly diagonalize but must instead try to preserve disagreeing options for the computation of \( X \) from \( \zeron{\beta} \Tjoin g \) and let the diagonalization occur on \( T_{\beta+1} \).

We first must ensure that \( T_{\beta+1} \) leaves options open that can be extended on \( T_\beta \) to incompatible computations.

\begin{definition}\label{def:padded-tree}
Say \( T_\beta \) is padded if whenever \( \sigma \in T_\beta \), \( \lh{\sigma} = 0 \mod 2 \)  then \( \sigma\concat[0], \sigma\concat[1]  \in T_\beta \).
\end{definition}

And now give conditions that ensure these incompatible computations exist.

\begin{definition}\label{def:disagreement-eager}
Say that \( \theta^{\beta+1}\map{T_{\beta+1}}{T_\beta} \) is disagreement preserving if whenever \( \sigma\concat[0], \sigma\concat[1] \in T_{\beta+1} \) and \( \lh{\sigma}=2i \geq \copylen{\beta} \) then either
	\begin{align*}
		& \forall[\tau, \tau' \supset \theta^{\beta+1}(\sigma)]( \recfnl{i}{\zeron{\beta}}{\tau} \compat \recfnl{i}{\zeron{\beta}}{\tau'} )\\
		\shortintertext{Or}
		& \recfnl{i}{\zeron{\beta}}{\theta^{\beta+1}(\sigma\concat[0])} \incompat \recfnl{i}{\zeron{\beta}}{\theta^{\beta+1}(\sigma\concat[1])}
	\end{align*}
\end{definition}

\begin{definition}\label{def:subgeneric-root}
Say that \( \seq{T_\beta}{\beta \kleeneleq \alpha} \) is a disagreement preserving downwardly generic tower if it is a downwardly generic tower  and for each \( \beta \) with \( \beta   \kleeneleq \alpha \), \( T_{\beta} \) is padded and \( \theta^{\beta+1} \) is disagreement preserving.  We define the notions of an \( \alpha \) subgeneric-reduct and \( \alpha \) subgeneric-root by modifying definition \ref{def:alpha-root} to use disagreement preserving downwardly generic towers.
\end{definition}

Note that we can produce \( \alpha \) subgeneric-reducts with the same degree of effectivity as we enjoyed for \( \alpha \) reducts.  

\begin{lemma}\label{lem:subgeneric-root}
If \( g \) is a \( \alpha \) subgeneric-root then \( g \) is \( \alpha \) subgeneric.
\end{lemma}

\begin{proof}
By lemma \ref{lem:alpha-subgeneric-part-1} it is sufficient to show that \( g \) satisfies part \ref{def:subgeneric:meet-zeron-beta} of definition \ref{def:subgeneric}.   Suppose, for a contradiction, that \( g \) fails this condition for the set \( X \) and let \( \beta \kleeneleq \alpha \) be the least such that \( X \Tleq \zeron{\beta} \) and for some \( \gamma \kleenel \beta \) we have \(  X \Tleq \zeron{\gamma} \Tjoin g \) but \( X \nTleq \zeron{\gamma} \).  Assume \( \beta \) is a successor then we must have \( \gamma +1 =\beta \) or \( \beta \) would not have been the least failure.  Now fix   \( e, j \) such that 
	\begin{equation}\label{lem:alpha-subgeneric:eq-counter}
	\begin{aligned}
		X &= \recfnl{e}{\zeron{\beta}}{}\\
		X &= \recfnl{j}{\zeron{\gamma} \Tplus g_{\gamma} }{}
	\end{aligned}
	\end{equation}
Let \( \psi \) be the   \( \Csigman[\zeron{\gamma +1}]{1} \) formula asserting that these computations disagree.  Since the computations agree we have \( g_{\gamma+1} \models \lnot \psi \)  so pick \( \sigma \subset g_{\gamma +1} \) such that \( 	\sigma \) super forces \( \lnot\psi \) on \( T_{\gamma+1} \) where

	\begin{equation}
		\psi = \exists({\sigma \subseteq g_{\gamma+1}})\exists({x,s})\left(\recfnl{e}{\zeron{\gamma +1}}{x}\conv[s] \neq \recfnl{j}{\zeron{\gamma} \Tplus \theta^{\beta}(\sigma) }{x}\conv[s]\right)
	\end{equation}
	
We now work to define an initial segment \( \sigma' \) of  \( g_{\gamma +1} \) extending \( \sigma \) so that \( T_{\gamma} \) will preserve any potential disagreement so it's observed between the inputs  \( \theta^{\gamma+1}(\sigma'\concat[0]) \) and \( \theta^{\gamma+1}(\sigma'\concat[1])] \) if ever.  For this purpose pick  \( j' > \max(\copylen{\gamma}, \lh{\sigma}) \) so that \( \recfnl{j}{}{} \cequiv \recfnl{j'}{}{} \) and set \( \sigma' = g_{\beta+1}\restr{2j'} \).  Now if \( \recfnl{j'}{\zeron{\gamma}}{\theta^{\gamma+1}(\sigma'\concat[0])} \incompat \recfnl{i}{\zeron{\gamma}}{\theta^{\gamma+1}(\sigma'\concat[1])} \) then for some choice of \( i\in \set{0,1}{} \) the string \(  \theta^{\gamma+1}(\sigma'\concat[i]) \) disagrees with \( X \) so  \( \sigma'\concat[i] \models \psi \) contradicting the assumption that \( \sigma \) super forced \( \lnot \psi \) on \( T_{\gamma+1} \).  Thus by definition \ref{def:disagreement-eager} every extension of \( \theta^{\gamma+1}(\sigma')=\tau \) yields compatible computations  under \( \recfnl{j'}{}{} \).  Thus given \( y \) we may compute \( X(y) \) as the value of the first converging computation \( \recfnl{j}{\zeron{\gamma} \Tplus \hat{\tau} }{y}  \) for some \( \hat{\tau} \in T_{\gamma}\) with \( \hat{\tau} \supseteq \tau \).  Since \( T_\gamma \Tleq \zeron{\gamma} \) this search can be performed computably in \( \zeron{\gamma} \) and as \( g_{\gamma} \supseteq \hat{\tau} \) eventually a converging computation will always be found.

Now suppose \( \beta \) is a limit and again let \( \gamma \kleenel \beta \) and \( e,j \) satisfy \eqref{lem:alpha-subgeneric:eq-counter}.  By the minimality of \( \beta \), if \( X \Tleq \zeron{\kappa} \) with  \( \beta \kleeneg \kappa \kleenegeq \gamma \) we are done so without loss of generality we may assume that \( \gamma=\kleenelim{\beta}{n} \).   To coordinate the behavior of the computation of \( X \) from the various \( g_\kappa \) we show that we can choose a single index for all such computations.

By lemma \ref{lem:tower-homeomorphism} we can uniformly compute \( g_\gamma=\treemap{\kappa}{\gamma}(g_{\kappa}) \) from \( g_{\kappa} \) using \( \zeron{\kappa}  \).  We claim that there is a single index \( j' \) such that \(X= \recfnl{j'}{\zeron{\kappa} \Tplus g_{\kappa} }{} \) whenever \( \beta \kleenegeq \kappa \kleenegeq \gamma \).  The computation coded by \( j' \) can check whether \( \lambda \kleeneleq \kappa \) by inspecting \( \zeron{\kappa} \)  allowing \( \recfnl{j'}{\zeron{\kappa}}{} \) to recover \( \kappa\) at which point it can apply \( \treemap{\kappa}{\gamma}(g_{\kappa}) \).  Note that our index \( j' \) has the property that \(  \recfnl{j'}{\zeron{\kappa+1}}{\sigma} = \recfnl{j'}{\zeron{\kappa}}{\theta^{\kappa+1}(\sigma)} \) provided \( \gamma \kleenel \kappa \kleenel \beta \) and \( \sigma \in T_{\kappa+1} \).  Armed with this index we define \( \phi \) asserting that some such computation for \( X \) disagrees with the computation from \( \zeron{\beta} \).
	
\begin{equation}
		\phi = \exists({\sigma \subseteq g_\beta})\exists({x,s,\kappa})\left(\beta \kleeneg \kappa \kleenegeq \gamma \land \recfnl{e}{\zeron{\beta}}{x}\conv[s] \neq \recfnl{j'}{\zeron{\kappa} \Tplus \treemap{\beta}{\kappa}(\sigma) }{x}\conv[s]\right)
\end{equation}
	
Since \( \phi \) is false let  \( \sigma \subset g_{\beta} \) on \( T_{\beta} \) force \( \lnot \phi \).   Now fix \( \hat{j} > \max(\copylen{\beta}+n, \lh{\sigma}) \) so that \( \recfnl{j}{}{} \cequiv \recfnl{j'}{}{} \cequiv \recfnl{\hat{j}}{}{} \) and let \( m = 2\hat{j}  - \copylen{\beta} \).  Now if \( \kappa =\kleenelim{\beta}{m} \) by \eqref{edef:copylen} \( \copylen{\kappa}=m + \copylen{\beta} = 2\hat{j} \) and by \eqref{edef:copylen-between} \( \copylen{\kappa'} > \copylen{\kappa} \) for \( \kappa \kleenel \kappa' \kleenel \beta \).  Now if \( \sigma' = g_{\beta}\restr{2\hat{j}}\) then we also have  \(\sigma'= g_{\kappa}\restr{2\hat{j}}=g_{\kappa+1}\restr{2\hat{j}} \) and \( \sigma'\concat[i] \in T_{\kappa+1}\restr{\copylen{\kappa+1}}=T_{\beta}\restr{\copylen{\kappa+1}} \).  Turning our attention back to definition \ref{def:disagreement-eager} suppose that  \( \recfnl{j'}{\zeron{\kappa}}{\theta^{\kappa+1}(\sigma'\concat[0])} \) and  \( \recfnl{j'}{\zeron{\kappa}}{\theta^{\kappa+1}(\sigma'\concat[1])} \) are incompatible then we can choose \( \tau=\sigma'\concat[i] \) such that  \(  \recfnl{j'}{\zeron{\kappa} \Tplus  \treemap{\beta}{\kappa}(\tau)} \) is incompatible with \( X \).  This holds as \( \treemap{\beta}{\kappa}(\tau) \) can be factored to \( \theta^{\kappa+1}(\treemap{\beta}{\kappa+1}(\tau))\) and as \( \lh{\tau}\leq\copylen{\kappa+1} \) simplifies to just \(  \theta^{\kappa+1}(\tau) \).  As such a \( \tau \) would force \( \phi \) we conclude that every \( \tau \in T_{\kappa} \) extending \( \theta^{\kappa+1}(\sigma') \) yields compatible computations and we compute \( X \) from \( \zeron{\kappa} \) as we did in the successor stages.  Hence \( X \Tleq \zeron{\kappa} \) contradicting the minimality of \( \beta \). 
\end{proof}

We now show that a slight modification of the construction of \( T_{\beta} \) from \( T_{\beta+1} \) and \( T_{\enumpred{\beta}} \) we performed above lets us build a subgeneric \( \alpha \) root.

\begin{lemma}\label{lem:tree-minus-one-alt1}
	The following conditions may be added to those of lemma \ref{lem:tree-minus-one} so that \( T_{\beta} \) continues to be effectively built from \( T_{\beta+1} \) and \( T_{\enumpred{\beta}} \) while jointly satisfying all conditions.

\begin{arabiclist}\setcounter{enumi}{3}
	\item If \( \gamma=\enumpred{\beta} \) and \( T_\gamma \) is padded or \( \copylen{\beta}=0 \)  then \( T_\beta \) is padded.
	\item \( \theta^{\beta+1} \) is disagreement preserving\label{lem:tree-minus-one:disagreement-preserving}
\end{arabiclist}
\end{lemma}

\begin{proof}
	To ensure that  part \ref{lem:tree-minus-one:disagreement-preserving} holds whenever \( \lh{\sigma}=2i \geq \copylen{\beta} \), \( \tau_0= \theta_s^{\beta+1}(\sigma\concat[0]) \) and \( \tau_1= \theta_s^{\beta+1}(\sigma\concat[1]) \) are both defined and \( \recfnl[s]{i}{\tau_0}{} \compat \recfnl[s]{i}{\tau_1}{} \), but there are \( \tau'_0, \tau'_1  \supset \theta_s^{\beta+1}(\sigma) \) with \( \recfnl[s]{i}{\tau'_0}{} \incompat \recfnl[s]{i}{\tau'_1}{} \) then set \( \theta_{s+1}^{\beta+1}(\sigma\concat[0])=\tau'_0 \) and \( \theta_{s+1}^{\beta+1}(\sigma\concat[1])=\tau'_1 \) and unset \( \theta_{s+1}^{\beta+1}(\tau) \) for every \( \tau \supsetneq \sigma \).
	To ensure that \( T_\beta \) is padded we simply place \( \sigma\concat[0] \) and \(  \sigma\concat[1] \) into \( T_\beta \) whenever \( \sigma \in T_\beta \) and for some \( l \) \( \lh{\sigma}= 2l \geq \copylen{\beta} \).  If \( \copylen{\beta}=0 \) and this suffices.  Otherwise \( T_\beta \) is padded as \( T_\gamma\restr{\copylen{\beta}}=T_{\beta}\restr{\copylen{\beta}} \) and \( T_\gamma \) is padded.
\end{proof}

Note that given an initial tree \( T \) we can easily perform effective modifications to ensure it is padded so substituting lemma \ref{lem:tree-minus-one-alt1} into the construction given by lemma \ref{lem:downward-tower-exists} yields  a disagreement preserving downwardly generic tower of length \( \alpha \).  Thus given a \( \zeron{\alpha} \) index for \( g_{\alpha} \) viewed as a tree \( T_{\alpha} \) with \( [T_{\alpha}]=\set{g_\alpha}{} \)  we can compute the index for a computable tree \( T_0 \) with a unique  path \( g \) which by lemma \ref{lem:subgeneric-root}  is \( \alpha \) subgeneric.  While this easily gives a (uniformly witnessed) sequence of \( \pizn{1} \) singletons \( g_i \) of \( \alpha \) subgenerics this is not quite sufficient to prove corollary \ref{cor:subgeneric} as we must still ensure that part \ref{cor:subgeneric:alpha-independence} of definition \ref{cor:subgeneric} holds.  To do this we observe

\begin{lemma}\label{lem:uniform-alpha-independence}
There is a uniform sequence \( T^{i}_{\alpha} \Tleq \zeron{\alpha} \) each having a unique path \( \hat{g_{i}} \) such that
	\[
	\hat{g_k} \nTleq \left(\TPlus_{i \neq k} \hat{g_i} \Tjoin \zeron{\alpha} \right)
	\]
\end{lemma}

\begin{proof}
	Our construction builds \( \hat{g_k}  \) as the limit of \( \hat{g_{k,s}}  \) via a finite injury argument.  The requirements \( \mathcal{R}_{k,i} \) demand that \(\hat{g_k} \neq \recfnl{i}{\TPlus_{i \neq k} \hat{g_i} \Tjoin \zeron{\alpha}}{} \) and are met by changing the value of \( g_{k,s+1}(x) \) to disagree with the computation in question whenever such a change is not restrained by a higher priority requirement and restraining any changes in the use of this computation or of \( g_{k}(x) \).  Every time \( g_{k,s+1}(x) \) is set to a new value it is picked large enough not yet to have been enumerated into the compliment of \( T^{i}_{\alpha}\).
\end{proof}

This now suffices to complete the proof of corollary \ref{cor:subgeneric}.  Using the uniformity of the trees \( T^{i}_{\alpha} \) from lemma
\ref{lem:uniform-alpha-independence} and the uniformity of the construction in lemma \ref{lem:downward-tower-exists} (using the modified lemma \ref{lem:tree-minus-one-alt1}) we get a uniform sequence of computable trees \( T^{i}_0 \) each containing a single \( \alpha \) subgeneric path \( g_i \).  To see that part \ref{cor:subgeneric:alpha-independence} of definition \ref{cor:subgeneric} holds observe that the uniformity of the trees \( T^{i}_{\alpha} \) and the uniform definition of the maps \( \treemap{i,\alpha+1}{0} \) guarantees the equivalence \( g_i \Tjoin \zeron{\alpha} \Tequiv \hat{g_i} \Tjoin \zeron{\alpha} \) holds uniformly.  Also by the uniformity of lemma \ref{lem:alpha-subgeneric-part-1} \( g_i \Tjoin \zeron{\alpha} \Tequiv \jumpn{g_i}{\alpha} \) holds uniformly so if part \ref{cor:subgeneric:alpha-independence} failed we would have the contradiction
\[
\TPlus_{i \neq k} \hat{g_i} \Tjoin \zeron{\alpha} \Tgeq \TPlus_{i \neq k} \jumpn{g_i}{\alpha} \Tgeq g_k \Tjoin \zeron{\alpha} \Tgeq \hat{g_k}
\]

It is worth remarking that the result claimed by Harrington in \cite{mclaughlins-conjecture} isn't actually lemma \ref{cor:subgeneric} but the substantially stronger version below.    

\begin{corollary}[Harrington \cite{mclaughlins-conjecture}]\label{cor:subgeneric-star}
	For each \( \alpha < \wck \) there is a sequence \( \seq{g_n \in \baire}{n \in \omega} \) so that for all \( n \in \omega \)
	\begin{enumerate}
		\item \( \displaystyle \Tdeg{g_n}  \) is \( \alpha \) subgeneric.
		\item \(\displaystyle  g_n \nTleq \jumpn{\left(\TPlus_{i \neq n} g_i\right)}{\alpha}  \)\label{cor:subgeneric-star:alpha-independence}.
		\item \( g_n \) is a the unique solution of a \( \pizn{1} \) formula the index for which is given uniformly in \( n \).
	\end{enumerate}
\end{corollary}

Corollary \ref{cor:subgeneric-star} replaces \hyperref[cor:subgeneric:alpha-independence]{claim \ref*{cor:subgeneric:alpha-independence} of corollary \ref*{cor:subgeneric}} which required that no \( g_n \) could be computable in the join of the \( \alpha \) jumps of the remaining \( g_i \) with the substantially stronger requirement that \( g_n \) not be computable in the \( \alpha \) jump of the join of the remaining \( g_i \).  Corollary \ref{cor:subgeneric-star} is true but we have only been able to prove the result by making some substantial modifications to the underlying framework which we sketch below.  

To establish \hyperref[cor:subgeneric-star:alpha-independence]{claim \ref*{cor:subgeneric-star:alpha-independence} of corollary \ref*{cor:subgeneric-star}} we introduce a notion of mutual genericity for a sequence of the reals \( g^{i}\) on the sequence of trees \( T^i \) where \( g^{i} \) is a path through \( T_i \).  In particular we define \( \bm{g} = \TPlus g^{i} \) to be the function where \( \bm{g}(\pair{i}{x})=g^{i}(x) \) and write \( \setcol{\bm{g}}{n}(x)\)  for  \( \bm{g}(\pair{n}{x}) \).  We further define \( \bm{T}=\TPlus T^{i} \) to be the tree consisting of those nodes \( \bm{\sigma} \) with \( \setcol{\bm{\sigma}}{n} \in T^n \) for every \( n \) and say that the sequence of singletons \( g^i \) is mutually \( \alpha \) generic on \( T^i \) if \( \bm{g} \) is \( \alpha \) generic on \( \bm{T} \).  We prove \ref{cor:subgeneric-star} by simultaneously building \( \omega \) many disagreement preserving downwardly generic towers consisting of the trees \( T^{i}_\beta \) for \( \beta \kleeneleq \alpha \) where \( g^{i}_\beta \) is the unique path through \( g^{i}_\beta \) and the trees \( T^{i}_\beta \) satisfy the obvious generalization of being eagerly generic to the notion of mutually eagerly generic.  We stipulate our coding function has the property that \( \pair{n}{x+1}  \) is always greater than \( \pair{n}{x} \) so if \( \bm{\sigma} \in \bm{T} \) we may assume that \( \bm{\sigma}= \TPlus \sigma^{i} \) with each \( \sigma^{i} \in \wstrs \) and all but finitely many of them equal to the empty string

Generalizing our previous construction we now build \( T^{i}_\beta \) as the image of \( \theta^{i, \beta+1} \Tleq \zeron{\beta+1} \) mapping \( \bm{T_{\beta+1}} \) to \( T^{i}_\beta \).  Naively one might imagine that we could straightforwardly carry out the same forcing construction we used previously but now applied to \( \bm{T_\beta} \) as in the standard (not localized to a tree) product forcing construction.  However, since we wish to maintain \( \bm{T_\beta}=\TPlus T_{\beta}^{i} \) so as to still produce a sequence \( g^{i} \) of \( \alpha \)-subgeneric roots such a simple argument won't suffice.  In particular by demanding that every path through \( \bm{T_\beta} \) extending \( \bm{\sigma} \) also pass through \( \bm{\tau} \supseteq \bm{\sigma} \) we would impose pruning on the factors \( T^{i}_\beta \) which would in turn force a pruning of \( \bm{T_\beta} \) above other nodes \( \bm{\sigma'} \) even when \( \bm{\sigma'}  \) is incompatible with \( \bm{\sigma} \) because we could still have \( \setcol{\bm{\sigma} }{i}=\setcol{\bm{\sigma'} }{i} \).  To avoid this difficulty we ensure that the paths in \( T^{i}_\beta \) carry with them the information about the paths in \( T^{j}_\beta, j \neq i \).  In particular we will ensure that if \( \bm{\sigma} \incompat \bm{\sigma'} \) then \( \theta^{i, \beta+1}(\bm{\sigma}) \incompat \theta^{i, \beta+1}(\bm{\sigma'})  \) for every \( i \).

As before we define \( \theta^{i, \beta+1} \) as the limit of a stagewise construction so we suppose that \( \theta^{i, \beta+1}_s(\bm{\sigma}^{-}) \) is defined and \( \lh{\bm{\sigma}} \leq s \) and define \( \theta^{i, \beta+1}_{s+1}(\bm{\sigma}) \) for each \( i \) exactly as we did in lemma \ref{lem:tree-minus-one}.  Note that if \( \pair{n}{x}=\lh{\bm{\sigma}} \) this has the effect of duplicating the usual action of \( \theta^{\beta+1}_{s+1} \) as it would apply to \( \setcol{\bm{\sigma}}{n} \) in the single tree \( T^{n}_{\beta+1} \) at \( \theta^{i, \beta+1}_{s+1}(\bm{\sigma}^{-}) \) on \( T^{i} \) for all \( i \in \omega \).  It is readily apparent that this definition ensures the required incompatibility property mentioned above.  We may now safely prune branches on \( \bm{T_\beta} \) without fear of inadvertent interference.

We now simply demand that if at some stage \( s > 2i \) we discover some \( \bm{\sigma'} \supseteq \bm{\sigma}  \), \( \lh{\bm{\sigma}}=2i \) with some finite (contiguous) initial segment of \( \TPlus \theta^{i, \beta+1}_s(\bm{\sigma'}) \) meeting \( \REset(\zeron{\beta}){i} \) then we redefine \( \theta^{i, \beta+1}_{s+1}(\bm{\sigma}) \) to be equal to \( \theta^{i, \beta+1}_s(\bm{\sigma'}) \) and unset \( \theta^{i, \beta+1}_s(\bm{\sigma'}) \) for every \( \bm{\sigma'} \supsetneq \bm{\sigma} \).  By our incompatibility property this can't have any effect on any node in \( \bm{T_\beta} \) incompatible with \( \TPlus \theta^{i, \beta+1}_s(\bm{\sigma'}) \).  At this point we may now appeal to the fact that \( g^{i}_{\beta+1} \) is the unique path through \( \bm{T_{\beta+1}} \) and argue that \( \theta^{i, \beta+1}(\bm{\sigma}) \)  extends to an infinite path on \( T^{i}_\beta \) if and only if \( \setcol{\bm{\sigma}}{n} \) extends to an infinite path on \( T^{n}_{\beta+1} \) for all \( n \).  A similar approach can be applied to demand a slightly modified version of disagreement preservation. While this construction doesn't yield a \( \zeron{\beta+1} \) computable homeomorphism from \( [T^{n}_{\beta+1}] \) to \( [T^{n}_{\beta}] \) (since some nodes would need to be mapped to many potential values) it does yield a \( \zeron{\beta+1} \) homeomorphism from \( [\bm{T_{\beta}}] \)  to \( [\bm{T_{\beta+1}}]  \) and this is sufficient to give both the desired mutual genericity as well as preserve the desired properties of the non-mutual construction.  This completes our sketch of corollary \ref{cor:subgeneric-star}.

Before we finish our discussion of Harrington's work in \cite{mclaughlins-conjecture} one final corollary is worth mentioning.

\begin{corollary}[Harrington \cite{mclaughlins-conjecture}]\label{cor:totally-subgeneric}
There is a non-empty \( \pizn{1} \) class whose members are \( \alpha \) subgeneric for every \( \alpha < \wck \).
\end{corollary}

\begin{proof}
	In the appendix we establish the existence of a \( \Pi^{1}_1 \) linearly ordered set of notations  \( I \) cofinal in a path through \( \kleeneO \) such that the functions \( \enumpred{\beta}, \copylen{\beta} \) are uniformly computable in \( \alpha \) on the set \( \set{\beta }{\beta \kleeneleq \alpha} \) for \( \alpha \in I \).  Now consider the predicate \( \Lambda_{\alpha} \) consisting of those sets  \( T \) coding computable trees that with \( T=T_0 \) in some disagreement preserving downwardly generic tower of length \( \alpha \) with \( T  \) not having well-founded height less than the notation denoted by \( \alpha \).
	It is easily checked that \( \Lambda_{\alpha} \) is \( \deltaOneOne \) uniformly  in \( \alpha \) as it is easily defined via number quantification over \( \zeron{\alpha+1} \) since the set of notations whose height is less than that of \( \alpha \) is uniformly computable from \( \zeron{\alpha+1} \).  Thus we may safely identify \( \alpha \in I \) with the \( \deltaOneOne \) index for \( \Lambda_\alpha \).  Moreover \( \hat{\alpha} \kleeneg \alpha \) entails \( \Lambda_{\hat{\alpha}} \subseteq \Lambda_{\alpha} \) and clearly \( \Lambda_{\alpha} \neq \eset \).  
	Now fix some \( \deltaOneOne \) subset \( H \) of \( I \) and consider \( \Isect_{\alpha \in H} \Lambda_\alpha \).  By \( \sigmaOneOne \) bounding there is some \( \alpha \in I \) bounding \( H \) so this intersection contains the non-empty collection \( \Lambda_\alpha \).   Thus by Kreisel's compactness theorem \cite{higher-recursion-theory} there is some \( T \) in \( \Isect_{\alpha \in \kleeneO} \Lambda_\alpha \) and by lemma \ref{lem:subgeneric-root} every path through \( T \) is \( \alpha \) subgeneric for every \( \alpha < \wck \).  Moreover, \( T \) doesn't have well-founded height below \( \wck \) so as \( T \) is computable there must be some infinite path through \( T \).  Indeed, since no path through \( T \) is hyperarithmetic \( [T] \) must be a perfect set.
\end{proof}

\section{Moduli of Computation}
\subsection{Background}

Interest in the computational properties of fast growing functions goes back to Post's program and the realization that how fast the enumeration of  \( \setcmp{M} \) grows is a measure of the thinness of \( \setcmp{M} \), and as Rice first showed \cite{Rice:rec-and-re-orders}  when he characterized the hyperimmune sets this way, it's often an easier concept to work with.   With seeming ingratitude this approach soon turned on Post giving Yates  \cite{Yates:ThreeREtheorems} the tools he needed to put the nails in the coffin of Post's Program by building a complete maximal set.  Later Martin improved this analysis to fully characterize the degrees of maximal sets \cite{Martin:FuncDom} and even today studying the relation between rate of growth and computational power continues to pay off \cite{almost-everywhere-domination,on-a-conjecture-of-dobrinen-and-simpson-concerning,almost-everywhere-domination-and-superhighness,low-for-random-reals-and-positive-measure-domination,mass-problems-and-almost-everywhere-domination,Uniform_A.E._Domination}.  Strangely, however, while many different notions of `fast growing' have been proposed and the degrees of such functions (partially) characterized little work has been done in the other direction.  That is given a degree how fast much a function grow to compute that degree?  To this end we follow Slaman and Groszek in introducing the following definitions \cite{SlamanModulus}.

\begin{definition}\label{def:moduli-of-computation}
	The function \( h \in \baire \) is a moduli (of computation) for a degree \( \Tdeg{d} \) if every \( f \fung h \) computes \( \Tdeg{d} \).  \( h \) is a uniform moduli of computation if there is some fixed computable functional \( \recfnl{}{}{} \) and \( D \in \Tdeg{d} \) such that \( f \fung h \implies \recfnl{}{f}{}=D \).  If furthermore \( h \) is a (uniform) moduli of computation for \( \Tdeg{h} \) we say that \( h \) is a (uniform) self-moduli.
\end{definition}

It's natural to respond to this definition by first asking when can a degree \( \Tdeg{d} \) even have a moduli of computation?  What about a uniform moduli of computation?  Can any degree be computed (uniformly?) by sufficiently fast growing functions?  Though this side of the relationship between rates of growth and computational power hasn't received as much attention as it's opposite these questions are natural enough they have multiple published solutions that are disguised by terminological differences.  We first look to the uniform case where one can show the degrees with a uniform self-moduli are just the \( \pizn{1} \) singletons (in \( \baire \) ) we direct the reader to Jockusch and McLaughlin \cite{countable-retracing-functions-and-pi02-predicates} for the earliest easily straightforward English language proof but follow them in crediting Kuznecov and Trahtenbrot \cite{investigation-of-partially-recursive-operators-by-means-of-the-theory} and latter Myhill \cite{finitely-representable-functions}.  We generalize this result to those functions \( h \) with a uniform moduli in some computable ordinal number of jumps.  Informally the relationship is simply that \( h \) is a \( \pizn{\beta+1} \) singleton if and only if the natural fast growing function computable in \( \jumpn{h}{\beta} \) is a uniform modulus for \( h \).  To state the theorem formally we need to replace ``natural fast growing function'' with an explicit function.

\begin{definition}\label{def:fastn}
\begin{align*}
	& \fastn{0}(x) = 0 \\
	& \hat{\xi}^{\beta +1}(x) = \begin{cases}
														\fastn{\beta}(x) & \text{if } x < \gcode{\beta+1}\\
														\min \set{t}{ \forall[i < x]\left(\recfnl{i}{\zeron{\beta}}{i}\conv \iff \recfnl{i}{\zeron{\beta}}{i}\conv[t] \right) } & \text{otherwise}
													\end{cases} \\
		& \fastn{\beta+1}(x)=\max(\hat{\xi}^{\beta +1}(x), \sup_{\gamma \kleeneleq \beta} \fastn{\gamma}(x)) \\
		\shortintertext{For \( \lambda \kleeneleq \alpha \) a limit}
	& \fastn{\lambda}(x) = \sup_{\beta \kleenel \lambda} \fastn{\beta}(x) = \sup_{\substack{ \beta \kleenel \lambda \\ \gcode{\beta} \leq x}} \fastn{\beta}(x)
\end{align*}
We relativize this notion by setting \(	\fastn(h){0}=h \) and building \( \fastn(h){\alpha} \) as above.  

\end{definition}
Note that as the notations below some given notation \( \beta \) can be effectively computed from \( \zeroj \) these supremums can be easily deciphered by any set computing \( \zeroj \).  With this in mind the following properties should be straightforward to verify so are presented without proof.

\begin{lemma}\label{lem:fastn}
	For each \( \beta \kleeneleq \alpha \), \( h \in \baire \) \( \fastn(h){\beta} \) is uniformly computable from \( \jumpn{h}{\beta} \) and has the following properties.
	\begin{arabiclist}
		\item \( \fastn{\beta} \) is a uniform self-modulus for \( \zeron{\beta} \) and the functional witnessing this uniformity is itself uniform in \( \beta \).

		\item \( \forall[\gamma \kleenel \beta]\left( \fastn(h){\gamma} \funl \fastn(h){\beta} \right) \).
		\item There is a stagewise approximation \( \fastn[s](h){\beta+1} \)  uniformly computable in \( \jumpn{h}{\beta} \) and strictly increasing in \( s \) such that \( \lim_{s\to\infty} \fastn(h)[s]{\beta+1}(x)=\fastn(h){\beta+1}(x) \).\label{lem:fastn:increasing}
		\item The predicate \( \fastn(h){\beta}(x) \geq y \) is \( \Csigman(h){\beta} \).\label{lem:fastn:sigman-geq}
	\end{arabiclist}
\end{lemma}

Note that all of the above relativizes to \( \fastn[][h]{\beta} \).  We can now formally describe the general relation between \( \pizn{\beta+1} \) singletons and uniform moduli.  

\begin{theorem}\label{thm:uniform-modulus-in-beta-jumps}
	\( \fastn(h){\beta} \) is a uniform modulus for \( h \) if and only if \( h \) is a \( \pizn{\beta+1} \) singleton.
\end{theorem}

\begin{proof}
	Suppose \( \fastn(h){\beta} \) is a uniform modulus for \( h \) witnessed by the reduction \( \recfnl{}{}{} \).  By part \ref{lem:fastn:sigman-geq} of lemma \ref{lem:fastn} we note that there is a \( \Csigman(h){\beta} \) formula \( \psi(\sigma) \) asserting that there is some \( x < \lh{\sigma} \) and \( \fastn(h){\beta}(x) > \sigma(x) \).  Thus the formula \( \phi(g) \) defined below is equivalent to a \( \Cpin{\beta+1} \) formula
\[
\phi(g)\eqdef\forall[\sigma \in \wstrs]\forall[s]\forall[x]\left(\recfnl[s]{}{\sigma}{x}\diverge \lor \recfnl[s]{}{\sigma}{x}=g(x) \lor \psi(\sigma)   \right)
\]
Clearly \( \phi(h) \) holds as if \( \sigma \fung \fastn(h){\beta} \) and \( \recfnl{}{\sigma}{x}\conv \) it must have value \( g(x) \).  But if \( \hat{h}\neq h \) then pick \( x \) with \( \hat{h}(x)\neq h(x) \) and \( \sigma =\fastn(h){\beta}\restr{l} \) where \( l > \use{\recfnl{}{\fastn(h){\beta}}{x}} \).  Such \( \sigma \) witnesses \( \lnot\phi(\hat{h}) \) so \( h \) is a  \( \pizn{\beta+1} \) singleton.

Conversely suppose that \( \phi \in \Cpin{\beta+1} \) has unique solution \( h \).  We may put \( \phi \) in the form \( \forall[\sigma \subset h]\psi(\sigma) \) where \( \psi \) is a \( \Csigman{\beta} \) formula not mentioning \( h \).  Now given \( g \fung \fastn(h){\beta}  \) we can uniformly recover \( \zeron{\beta} \) from \( g \) and thus compute a tree \( T \) consisting of all those \( \sigma \funl g \) for which \( \psi(\sigma) \) holds.  As \( h \funl \fastn(h){\beta} \funl g \) \( h \) is a path through \( T \) and as any path through \( T \) would satisfy \( \phi \) it is unique.  As \( T \) is finitely branching we can avail ourselves of K\"{o}enig's lemma to establish that \( g \) uniformly computes \( h \).
\end{proof}

Note that the specific form of \( \fastn(h){\beta+1} \) isn't important only that \( \fastn(h){\beta+1} >> h \) and that \( \fastn(h){\beta+1} \) is a uniform modulus for  \( \zeron{\beta+1} \).

This suffices to give a uniform modulus for every hyperarithmetic function and it is easy to see (as in \cite{SlamanModulus}) that \( h \in \baire \) has a uniform modulus if and only if \( h \) is \( \deltaOneOne \).  While Solovay finally classified those functions with some modulus of computation in \cite{HyperEncodable} using a different method we follow the approach taken in \cite{SlamanModulus} using Hechler style forcing conditions to demonstrate the following lemma.

\begin{lemma}[Slaman and Groszek]\label{lem:mod-imp-uniform}
If \( g \) has a modulus of computation  than \( g \) has a uniform modulus of computation
\end{lemma}

We quickly sketch the proof.  The conditions will be Cohen style conditions in \( \wstrs \) paired with some \( q \in \baire \) we commit to majorizing. If \( h \) is a modulus for \( g \) then the forcing conditions do their best to produce some \( f \fung h \) not computing \( g \) and their failure can only occur if there is some sufficiently fast growing \( \hat{h} \fung h \) above which the reduction is uniform.  Combining this result with the remark above yields Solovay's result.

\begin{theorem}[Solovay]\label{thm:modulus-hyper}
\( h \in \baire \) has a modulus of computation if and only if \( h \) is \( \deltaOneOne \).
\end{theorem}

\subsection{Non-uniform Moduli}

The above results pose a very puzzling question:  All the natural examples of moduli are uniform moduli and every function with a modulus must have a uniform modulus so can the two notions come apart and if so by how much?  The remainder of this paper is devoted to showing that these two notions come apart as far as possible.  In particular we prove the following result.

\begin{theorem}\label{thm:nonu}
For each \( \alpha \in \kleeneO \) there is a a self-modulus \( \nusm \Tleq \zeron{\alpha} \) such that no \( f \ttleq \jumpn{\nusm}{\beta} \) for any \( \beta \kleenel \alpha \) is a uniform modulus for \( \nusm \).
\end{theorem}

We can now prove \ref{thm:nonu}.  We proceed by fixing some notation \( \alpha \) and describe in this section how to build a self-modulus \( \nusm \) with no uniform moduli computable from any \( \zeron{\beta}, \beta \kleenel \alpha \).  This requires walking a careful line between making \( \nusm \) unique enough that every faster growing function computes it but not so unique that they can do so uniformly.  Our approach
is to build \( \nusm \) as a highly `generic' function  that is nevertheless unique for all `small' functions majorizing it.  Any `large' function majorizing \( \nusm \) will have enough computational power to watch our construction of \( \nusm \) while the uniqueness of \( \nusm \) relative to the `small' functions majorizing \( \nusm \) will let them compute \( \nusm \).  Essentially large will mean dominating \( \fastn{\alpha} \) and \( \nusm \) will be built computably in \( \zeron{\alpha} \) leaving the rest of the construction to deal with small functions and to be sufficiently generic to avoid small uniform moduli.  The level of genericity required is given by the next lemma.

\begin{lemma}\label{lem:generic-not-unif-sm}
If \( g \) is \( \alpha \) generic on \( T \Tleq \Tzero \) and non-isolated then no \( h \ttleq \jumpn{g}{\beta}, \beta \kleenel \alpha \) is a uniform modulus for \( g \).
\end{lemma}
\begin{proof}
	For contradiction fix \( g, h \) as in the lemma,  \( \Psi \) a truth-table functional and \( \Phi \) a Turing functional such that \( h=\Psi(\jumpn{g}{\beta}) \) and for every \( f \fung h \) \( \Phi(f)=g \).  Now let \( \phi \) be the \( \Csigman{\beta+1} \) formula defined below asserting that for some \( \sigma \fung h \)  \( \Phi(\sigma) \) disagrees with \( g \).
	\begin{equation}
		 \phi\eqdef\exists({\tau \fung \Psi(\jumpn{g}{\beta})})\exists({x})\left(\Phi(\tau;x)\conv[\lh{\tau}] \neq g(x)\right)
	\end{equation}
	If \( \phi(g) \) then there would be some \( f \fung \Psi(\jumpn{g}{\beta})=h \) extending \( \tau \) so \( \Phi(f)\neq g \).  Thus \( \phi(g) \) and as \( \beta+1 \kleeneleq \alpha \) we must have some \( \sigma \subset g \) with \( \sigma \frc[T] \phi \).  As \( g \) is non-isolated we can fix another path \( \hat{g} \neq g \) on \( T \) extending \( \sigma \).  Let \( \hat{h}=\Psi(\jumpn{\hat{g}}{\beta}) \) which as \( \Psi \) is a truth table reduction must be total.  Since \( \hat{g} \supset \sigma \) we know \( \phi(\hat(g)) \) holds.  Now fix some \( f \fung h, \hat{h} \).  By assumption \( \Phi(f)=g\neq \hat{g} \) contradicting the fact that \( \Phi(f) \) agrees with \( \hat{g} \) everywhere both are defined.
\end{proof}

Thus, we can make \( \nusm \) sufficiently generic for our purposes by building it as a non-isolated path through some \( \alpha \) reduct \( T_0 \).  If we had simply made \( \nusm \) fully \( \alpha \) generic then it wouldn't be a self-modulus at all since if \( g \) is a non-isolated path through \( T \) and \( 2 \)-generic on \( T \) \textit{relative} to \( T \) then \( g \) is not a self-modulus.  Ideally we would simply manipulate \( T_0 \) so that if some \( h \fung f,g \in [T_0] \) then \( h \fung \fastn{\alpha} \) ensuring that if \( h \fung \nusm \) either \( \set{\sigma \in T_0}{\sigma \funl h} \) has unique path \( \nusm \) or \( h \Tgeq \zeron{\alpha} \Tgeq \nusm \).  However, \( T_0 \) must be computable so this condition is too strong.  Instead we will impose a scrambled version of this condition.

\begin{definition}\label{def:uniquely-small}
	A tree \( T \) is uniquely \( \beta \) small if
	\begin{equation}
		\forall[{f,g \in [T]}]\forall[x]\left( f(x), g(x) \leq \fastn{\beta}(x) \implies f\restr{x+1}=g\restr{x+1}   \right)
	\end{equation}
\end{definition}

Ultimately we must project the impact of making \( T_{\beta+1} \) uniquely \( \beta \) small down to \( T_0 \) without while retaining the ability to extract \( \zeron{\beta+1} \) from \( \zeron{\beta} \) and \( h \) where \( h \fung f,g \) for some \( f,g \in [T_0] \).  This requires we further restrict our choice of reduction functions \( \theta^{\beta+1} \).

\begin{definition}\label{def:largness-preserving}
Say a monotonic map \( \theta\map{T^{*}}{T} \) is largeness preserving if every \( \sigma' \in \rng \theta \) is non-decreasing and if \( \theta(\sigma\concat[i])=\sigma'\concat\tau' \) then every \( x \in \dom \tau' \) satisfies \( \tau'(x) \geq i \).
\end{definition}

\begin{definition}\label{def:uniquely-small-tower}
Say a downwardly generic tower \(  \seq{T_\beta}{\beta \kleeneleq \alpha}  \) is a uniquely small tower (of length \( \alpha \)) if every \( T_\beta \) is uniquely \( \beta \) small and every \( \theta^{\beta+1} \) in part \ref{def:downwardly-generic:map} of definition \ref{def:downwardly-generic:map} is largeness preserving.  Say \( T_0 \) is an \( \alpha \) uniquely small reduct  if it occurs in some uniquely small tower of length \( \alpha \).
\end{definition}

We now fix a uniquely small tower \(  \seq{T_\beta}{\beta \kleeneleq \alpha}  \) of length \( \alpha \) for the remainder of the proof. 

\begin{lemma}\label{lem:treemap-largeness-preserving}
	For every \( \beta \kleeneleq \alpha \) the monotonic function \( \treemap{\beta}{0} \) is largeness preserving.
\end{lemma}
\begin{proof}
	Suppose that \( \beta \) is the least failure.  If \( \beta=\gamma +1 \) then \( \treemap{\beta}{0}=\treemap{\gamma}{0} \compfunc \theta^{\gamma+1} \).
	It is straightforward to verify that the composition of two largeness preserving functions is largeness preserving yielding the contradiction.  Now suppose \( \beta \) is a limit.  Since \( \theta^{1} \) is largeness preserving we know that \( \treemap{\beta}{0} \) is non-decreasing.  But \( \treemap{\beta}{0}(\sigma\concat[i]) = \treemap{\gamma}{0}(\sigma\concat[i]) \) for some \( \gamma \kleeneleq \beta \)  by part \ref{lem:tower-homeomorphism:limits} of lemma \ref{lem:tower-homeomorphism} so by the minimality of \( \beta \) \( \treemap{\beta}{0}(\sigma\concat[i]) \) can't fail the other half of definition \ref{def:largness-preserving}.
\end{proof}

While the motivation for making \( \theta^{\beta+1} \) largeness preserving is to protect the encoding of \( \zeron{\alpha} \) in every pair of \( f, g \in [T_0] \) it also provides the following useful property.

\begin{lemma}\label{lem:all-preimages-below}
A bound \( l(n) \) on \( \lh{\treemap{\beta}{}(\sigma)} \) for those \( \sigma \) with \( \lh{\sigma} < n \) and \( \treemap{\beta}{}(\sigma) \funl h \) can be uniformly computed from \( \zeron{\beta} \Tplus h \).
\end{lemma}
\begin{proof}  
	Let \( S_0 \) be the set containing the empty string and let \( l(0)=\lh{\treemap{\beta}{}(\estr)} \).
	\begin{align*}
		S_{n+1} &= \set{\sigma\concat[i]}{\sigma \in S_n \land \exists({x \leq l(n)+1})\left(h(x)\geq i\right) \land \sigma\concat[i] \in T_{\beta+1}  }\\
		l(n+1) &= \max_{\sigma \in S_{n+1}} \lh{\treemap{\beta}{}(\sigma)}
	\end{align*}
	Note that by definition \ref{def:largness-preserving} in defining \( S_{n+1} \)we've only excluded values of \( i \) that guarantee \( h \nfung \treemap{\beta}{}(\sigma\concat[i]) \).
\end{proof}

\begin{lemma}\label{lem:singleton-or-jump}
Suppose \( \nusm \Tleq \zeron{\alpha+1} \) is a path through an \( \alpha \) uniquely small reduct  \( T_0 \) then \( \nusm \) is a self-modulus.
\end{lemma}

\begin{proof}
Let \( h \fung \nusm \) which without loss of generality we may assume is non-decreasing.  First suppose that \( \nusm \) is the only path through \( T_0 \) satisfying \( h \fung \nusm \).  In this case let \( \hat{T} \) be the set of \( \sigma \in T_0 \) with \( h \fung \sigma \).  Clearly \( \hat{T} \) is a tree and \( \nusm  \) is the unique path through \(  \hat{T} \).  As \( \hat{T} \) is a finitely branching tree computable in \( h \), K\"{o}enig's lemma lets us uniformly compute  \( \nusm \) from \( h \).  So suppose \( \nusm,g \in T_0 \) with \( h \fung \nusm,g \) and \( \nusm(y) \neq g(y) \).    We argue by effective transfinite recursion that \( h \Tgeq \zeron{\beta} \) (uniformly in \( \beta \)) for \( \beta \kleeneleq \alpha + 1 \) leaving the routine details for the reader.  At limit stages the induction is straightforward so suppose \( h \Tgeq \zeron{\beta} \).

	We show that given \( x > y\)  \( h \) can (uniformly) compute  \( f(x) \geq \fastn{\beta+1}(x) \) and thus \( h \Tgeq f \Tgeq \zeron{\beta+1} \).    By way of lemmas \ref{lem:treemap-largeness-preserving} and \ref{lem:all-preimages-below} we compute \( l \) such that if \( \lh{\sigma} \leq x+1 \) and \( h \fung \treemap{\beta}{0}(\sigma) \) then \( \lh{\treemap{\beta}{0}(\sigma)} < l \).   We now verify \( f(x) \geq \fastn{\beta+1}(x) \).  

	Since \( g_\beta, \nusm_\beta \in [T_{\beta}]  \) and by monotonicity \( g_\beta\restr{x+1} \neq \nusm_\beta\restr{x+1} \) either \( g_\beta(x) \geq \fastn{\beta+1}(x) \) or \( \nusm_\beta(x)  \geq \fastn{\beta+1}(x) \).  But \( \lh{\treemap{\beta}{}(g_\beta\restr{x+1})} < l \) and \( \lh{\treemap{\beta}{}(\nusm_\beta\restr{x+1})} < l \) so as \( \treemap{\beta}{} \) is largeness preserving and \( h \fung g, \nusm\) is monotonic we can define \( f(x)=h(l) \).    
	This completes the proof that \( \nusm \) is a self-modulus.  
\end{proof}

We fill in the final piece of the puzzle by embellishing our construction from lemma \ref{lem:tree-minus-one} so that the resulting downwardly generic tower is a uniquely small tower.

\begin{lemma}\label{lem:tree-minus-one-small}
The statement of lemma \ref{lem:tree-minus-one} still holds if we also demand that \( \theta^{\beta+1} \) is largeness preserving and \( T_{\beta} \) is uniquely \( \beta \) small.
\end{lemma}
\begin{proof}
We sketch the modifications the proof of lemma \ref{lem:tree-minus-one} requires.  Since it is trivial to ensure \( \theta^{\beta+1} \) is largeness preserving simply by restricting which nodes we consider as values for \( \theta_s^{\beta+1} \) we restrict our attention to ensuring that \( T_\beta \) is uniquely \( \beta \) small.  

It is easy to eventually recognize pairs of nodes \( \sigma, \tau \) with \( \lh{\sigma}=\lh{\tau}=x+1 \) in  \( \rng \theta_s^{\beta+1} \) such that \( \sigma(x), \tau(x) < \fastn[s]{\beta+1}(x) \) and to abandon (remove from \( \rng \theta_{s+1}^{\beta+1} \)) one or the other rendering it a terminal branch.  Provided we always cut off one member of any such pair \( T_\beta \) will surely be uniquely \( \beta \) small.  The difficulty lies only in ensuring we choose the correct nodes to cut so as not to collaborate with our attempts to make \( T_\beta \) eagerly generic in a way that prunes all infinite paths from \( T_\beta \). 

Our solution is to work in the domain rather than the image and regard \( \theta_s^{\beta+1}(\sigma) \) to have priority \( \gcode{\sigma} \).  Note that we assume that \( \sigma \subset \tau \)  implies that \( \gcode{\sigma} \leq \gcode{\tau} \).   If at the end of stage \( s \) we discover some minimal pair of strings \( \sigma, \tau \) and \( x \) with \( \theta_s^{\beta+1}(\sigma)(x), \theta_s^{\beta+1}(\tau)(x) < \fastn[s]{\beta+1}(x) \) where \( \sigma \neq \tau  \) with \( \gcode{\tau} < \gcode{\sigma} \) we set \( \theta_s^{\beta+1}(\sigma) \) to be undefined.  Note that by working in the domain if we act to meet some genericity requirement by forcing all extensions of \( \theta_s^{\beta+1}(\sigma) \) to pass through \( \tau \) the only way \( \tau \) could later be pruned from the tree is if \( \theta_s^{\beta+1}(\sigma) \) is pruned from the tree so our additional pruning can't stop us from making \( T_\beta \) eagerly generic.   We now argue that if \( \sigma \in T_{\beta+1} \) then eventually \( \theta_s^{\beta+1}(\sigma) \) settles down to a node that never gets pruned.
	
Assume that \( \sigma \) is the node on \( T_{\beta+1} \) with least code at which the claim fails, \( s \) is a stage large enough that for every \( \sigma' \) with with \( \gcode{\sigma'} < \gcode{\sigma} \) \( \theta_s^{\beta+1}(\sigma') \) has settled on it's final value, \( \sigma \) never leaves \( T_{\beta+1} \) after stage \( s \) and that \( l > \lh{\theta_s^{\beta+1}(\sigma')} \) for all such \( \sigma' \).   Let \( s' > s \) be larger than \( \fastn{\beta+1}(x) \) for all \( x \leq l \) and \( s'' > s' \) the first stage following \( s' \) at which  \( \theta_{s''}^{\beta+1}(\sigma) \) is reset.  If \( t > s'' \) then \( \theta_{s''}^{\beta+1}(\sigma) \fung \fastn{\beta+1}\restr{l} \) so after \( s'' \) \( \theta_{t}^{\beta+1}(\sigma) \) is never again reset on account of avoiding simultaneous smallness and \( \sigma \) never leaves \( T_{\beta+1} \) after \( s \) so \( \theta_{t}^{\beta+1}(\sigma) \) isn't reset after \( s'' \) on account of the homeomorphism requirements so it must be reset at some \( s'''> s'' \) on account of the genericity requirement.  But now nothing can reset \( \theta_{t}^{\beta+1}(\sigma) \) if \( t > s''' \).  Contradiction.
	
\end{proof}

This enough to complete our proof of theorem \ref{thm:nonu}.  Invoke lemma \ref{lem:tree-minus-one-small} to produce a a uniquely small tower \(  \seq{T_\beta}{\beta \kleeneleq \alpha}  \) with \( T_\alpha \Tleq \zeron{\alpha} \) homeomorphic to \( \wstrs \) with \( T_\alpha=T^{*} \) and set \( T=T_0, g=\treemap{\alpha}{0}(g^{*}) \).  Lemma \ref{lem:totally-generic} tells us that \( g \) is \( \alpha \)-generic on \( T \) and as \( g \) is non-isolated lemma \ref{lem:generic-not-unif-sm} guarantees that no \( h \ttleq \jumpn{g}{\beta}, \beta \kleenel \alpha \) is a uniform modulus for \( g \).  On the other hand as both \( g^{*} \) and \( \treemap{\alpha}{0} \) are computable in \( \zeron{\alpha} \) so is \( g \) and by lemma \ref{lem:singleton-or-jump} this entails that \( g \) is a self-modulus.

\subsection{Functions Lacking Simple Moduli}

Since every \( \deltaOneOne \) function has a moduli one might also be prompted to ask whether there are \( \deltaOneOne \) degrees that are far from any modulus.  Intuitively this should be true for sufficiently generic degrees and we verify this is the case.

\begin{theorem}\label{thm:distant-modulus}
	If \( g \) is a non-isolated path through \( T \) and \( g \) is \( \alpha \) generic on \( T \) \textit{relative} to \( \pruneTree{T} \) then no \( h \ttleq \jumpn{(g \Tplus \pruneTree{T})}{\beta} \) for some \( \beta \kleenel \alpha \) is a modulus for \( g \) so in particular no \( h \ttleq \jumpn{g}{\beta}  \) is a modulus for \( g \).
\end{theorem}
\begin{proof}
Let \( g \in [T] \) be \( \beta+1 \kleeneleq \alpha \) generic on \( T \) relative to \( \pruneTree{T} \) and let \( \Psi \) be a truth table functional, i.e. a Turing functional total on all inputs.  Suppose for a contradiction that \( m^g=\Psi(\jumpn{(g \Tplus \pruneTree{T})}{\beta}) \), is a modulus for \( g \).  We work to build some \( f \fung m^{g} \) not computing \( g \).  Let \( \sigma_0=\estr \) and \( h_0=m \).  At stage \( n+1 \) we define \( \sigma_n \supseteq \sigma_0 \) such that \( \sigma_n \fung h^{g}_{n+1} \) and \( h^{g}_{n+1} \ttleq \jumpn{(g \Tplus \pruneTree{T})}{\beta} \) with \( h^{g}_{n+1} \fung h^{g}_{n}  \).  Thus at each stage we commit to some initial segment of \( f \) and a function that \( f \) must majorize.  Our goal will be to force every Turing reduction from \( f \) either to disagree with \( g \) or to be partial.

If there is a \( \sigma \supseteq \sigma_n \) with \( \sigma \fung h^{g}_n \) forcing a disagreement between \( \recfnl{n}{\sigma}{} \) and \( g \) then let \( \sigma_{n+1}=\sigma \) and \( h^{g}_{n+1}=h^{g}_n \).  Otherwise if there is some \( h^{g} \ttleq \jumpn{(g \Tplus \pruneTree{T})}{\beta} \) with \( h^g \fung h^{g}_{n}  \) and an integer \( x \) such that no \( \sigma \supseteq \sigma_n \) with \( \sigma \fung h \) satisfies \( \recfnl{n}{\sigma}{x}\conv \) then leave \( \sigma_{n+1}=\sigma_n \) and set \( h_{n+1}=h \).  If one of these two alternatives is satisfied for every \( n \) then \( f=\Union_{n \in \omega} \sigma_n \) majorizes \( m^{g} \) but fails to compute \( g \) contradicting the assumption.  So suppose that for \( \sigma=\sigma_n \) and \( h=h_n \) neither alternative is satisfied.

We note that since \( g \) is non-isolated, computably in \( \jumpn{(g \Tplus \pruneTree{T})}{\beta}  \) one can enumerate an infinite list of distinct \( \beta \) generic branches  \( g_i \) of \( T \)  with \( g_i\restr{i}=g\restr{i} \) and \( h^{g_i}\restr{\lh{\sigma}}=h^{g}\restr{\lh{\sigma}} \) with the later property guaranteed simply by letting \( g_i \) equal \( g \) on a long enough initial segment.  Therefore we define \( \hat{h}^{g} \Tleq \jumpn{(g \Tplus \pruneTree{T})}{\beta} \) so that \( \hat{h}^{g}(x) \) searches for \( l \) and \( y \) with \( g\restr{l} \frc[T] h^{g}(x)=y \) and then searches for values \( y_i \) for each \( i \leq l \) such that \( g_i \frc[T] h^{g_i}(x)=y_i  \) and returns a number larger than \( y \) and all the \( y_i \).  Such values \( y_i \) must exist since \( h^{g_i} \) is total.  Thus \( \hat{h}^{g} \fung h^{g} \), \( \hat{h}^{g} \fung h^{g_i} \) and \( \sigma \fung \hat{h}^{g} \).  

Since \( g \) is \( \beta+1 \) generic relative to \( T \) on \( T \) there is some \( l \) such that \( g\restr{l} \) forces the \( \lnot \phi \) where \( \phi \) is the \( \Csigman[T]{\beta} \) property that some extension  \( \sigma' \supseteq \sigma \) with \( \sigma' \fung h^{g} \) disagrees with \( g \).  Fix \( x \) so that \( g(x) \neq g_l(x) \) and pick \( \tau \fung \hat{h}^{g} \) extending \( \sigma \) such that \( \recfnl{n}{\tau}{x}\conv \).  Such a \( \tau \) must exist as otherwise \( \hat{h}\) would have been a valid choice for \( h_{n+1} \).  Moreover \( \recfnl{n}{\tau}{x}=g(x)\) since \( g \) forced \( \lnot\phi \) and \( \tau \fung \hat{h}^{g} \fung h^{g}  \).  But as \( \hat{h}^{g} \fung h^{g_l} \) and \( g(x)\neq g_l(x) \) we have \( g_l \models \phi \) and as \( g_l \) is \( \beta \) generic on \( T \)  some \( \upsilon \in \pruneTree{T}\) with \( \upsilon  \supseteq g\restr{l} \) forces \( \phi \) contradicting the fact that \( g\restr{l} \) forced \( \lnot\phi \).
\end{proof}

The similarity between this result and lemma \ref{lem:generic-not-unif-sm} is striking.  Since lemma \ref{lem:generic-not-unif-sm} was in some sense a reflection of the fact that \( g \) has a uniform modulus truth table computable within \( \beta \) jumps if and only if \( g \) is a \( \pizn{\beta+1} \) singleton this naturally raises the following open question.

\begin{question}
Can the self-moduli be characterized in terms of definability like the characterization of the uniform self-moduli as the \( \pizn{1} \) singletons?
\end{question}

In particular we conjecture that \( g \) is a self-modulus if and only if there is a  computable tree \( T \) with \( g \in [T] \) and a \( \pizn{1} \) formula \( \phi({\pruneTree{T}}) \) such that for all \( \hat{T} \subseteq T \) with \( g \in [\hat{T}] \) the formula \( \phi({\pruneTree{\hat{T}}}) \) is uniquely satisfied by \( g \).  Ideally, however, there would be a simpler statement expressing the definability of the self-moduli.


\appendix

\section{Defining \( \copylen{\beta} \)}\label{sec:ord-org}

Here we make good on our promise to define \( \copylen{\beta} \) and \( \enumpred{\beta} \).  

\begin{definition}\label{def:notation-path}
A path from a limit notation  \( \lambda \) to  \( \beta \kleenel \lambda  \) to is a sequence \( \vec{\gamma}=\left(\gamma_0,\gamma_1,\ldots,\gamma_{n-1},\gamma_n \right) \) of notations such that \( \gamma_0 = \lambda, \gamma_n = \beta  \) and for every \( i \)  \( \gamma_{i+1} \) appears in the effective limit for \( \gamma_i \), i.e., 
	\[  \forall[i < n]\exists({m})\left( \gamma_{i+1} = \kleenelim{\gamma_{i}}{m}\right)  \]
A path \( \vec{\gamma} \) from \( \lambda \) to \( \beta \) is minimal if \( \gamma_{i+1}=\kleenelim{\gamma_i}{m} \) for the least \( m \) such that \( \kleenelim{\gamma_{i+1}}{m} \kleenegeq \beta \).  
\end{definition}

Given an initial segment \( I \) of ordinal notations for every \( \beta \in I \) let  \( \enumpred{\beta} \) denote the least limit notation \( \lambda \kleeneg \beta \) in \( I \) with \( \beta = \kleenelim{\lambda}{m} \) for some \( m \).  If no such notation exists we write \( \enumpred{\beta}=\diverge \)

\begin{definition}\label{def:nice-ordinal}
A set \( I \) of ordinal notations is nice if \( I \) is a linearly ordered initial segment of \( \kleeneO \) such that for any \( \beta \kleenel \lambda  \in I \):
		\begin{arabiclist}
			\item There is at most one path denoted \( \enumseq[\lambda]{\beta} \) from \( \lambda \) to \( \beta \) and that path is minimal. \label{def:nice-ordinal:min}
			\item \( \enumseq[\lambda]{\beta} \) is defined whenever \( \kleenelim{\lambda}{0} \kleeneleq \beta \kleenel \lambda \).\label{def:nice-ordinal:between}
			\item There is no infinite  sequence \( \seq{\kappa_i}{i \in \omega} \) in \( I \) such that for all \( i \) \( \enumseq[\kappa_{i+1}]{\kappa_i} \) is defined.\label{def:nice-ordinal:finite-sequence}
			
		\end{arabiclist}
We say an ordinal \( \alpha \) is nice if the set \( \set{\beta}{ \beta \kleeneleq \alpha} \) is nice.
\end{definition}

	\begin{lemma}\label{lem:build-copylen}
If \( I \) is a nice set of notations then there is a total function \( \copylen{\beta} \) on \( I \) computable on every bounded initial segment \( I\restr{\alpha} \) satisfying:
		\begin{subequations}\label{eq:copylen-reqs}
		\begin{align}
				& \copylen{\beta} = \begin{cases}
				 												0 & \text{if } \enumpred{\beta}=\diverge\\
																\copylen{\enumpred{\beta}}+n & \text{if }  \beta=\kleenelim{\enumpred{\beta}}{n} 
														\end{cases}\\
				&  \kleenelim{\lambda}{n} \kleenel \beta \kleeneleq \lambda \implies \copylen{\kleenelim{\lambda}{n}} < \copylen{\beta}  \label{lem:build-copylen:inbetween}\\
			 & \lim_{n \to \infty} \copylen{\kleenelim{\lambda}{n}} = \infty \label{lem:build-copylen:infty} 
		\end{align}
	\end{subequations}
		Moreover \( \enumpred{\beta} \) is also computable on any initial segment \( I\restr{\kleeneleq \alpha} \).
		
	\end{lemma}
	\begin{proof}
We first note that \( \copylen{\beta} \) is always finite as otherwise the sequence \( \str{\beta, \enumpred{\beta}, \enumpred{\enumpred{\beta}}, \ldots } \) would violate \hyperref[def:nice-ordinal:finite-sequence]{part \ref*{def:nice-ordinal:finite-sequence} of definition \ref*{def:nice-ordinal}}.  To verify \eqref{lem:build-copylen:inbetween} note that by applying the minimality of paths in \( I \) we know that \( \enumseq[\lambda]{\beta} \) passes through \( \kleenelim{\lambda}{m} \) for some \( m > n \).   Claim \eqref{lem:build-copylen:infty} now follows trivially by considering the paths from \( \lambda \) to \( \kleenelim{\lambda}{n} \).  It remains only to show the computability.  

Now we show if \( I \) has a maximal element \( \alpha \) then \( \copylen{\beta} \) is a computable function on \( I \).  To this end we define a decreasing sequence of ordinals \( \kappa_i \) dividing \( I \) into connected pieces.  Let \( \kappa_0=\alpha \) and if \( \kappa_i \) is a limit ordinal set  \( \kappa_{i+1} \) to be the predecessor of the least \( \beta  \) with \( \enumseq[\kappa_i]{\beta} \) defined.   If \( \kappa_i \) is a successor let \( \kappa_{i+1} \) be the predecessor of \( \kappa_i \).  Since this is a decreasing sequence of ordinals it must be finite thus for some \( n \) \( \kappa_n=0 \) and by definition \ref{def:nice-ordinal} for every \( \beta \in I \) there is exactly one \( \kappa_i \) with \( \enumseq[\kappa_i]{\beta} \) defined.  We may compute \( \copylen{\beta} \) by enumerating all paths from some \( \kappa_i \) until we find some path \( \enumseq[\kappa_i]{\beta} \) at which point we may set \( \copylen{\beta}=\norm{\enumseq[\kappa_i]{\beta}} \).

If \( I \) lacks a maximal element we note there is an increasing sequence \( \alpha_i \) cofinal in \( I \) such that \( \enumpred{\alpha_i} \) isn't defined for any \( i \).  By part \hyperref[def:nice-ordinal:finite-sequence]{\ref*{def:nice-ordinal:finite-sequence} of definition \ref*{def:nice-ordinal}} we can build \( \alpha_i \) from any increasing cofinal sequence by repeatedly applying the operation taking \( \beta \) to \( \enumpred{\beta} \) until no longer possible.  Since the definition for \( \copylen{\beta} \) in \( I \) and \( I\restr{\alpha_i+1} \) agree when \( \beta \in I\restr{\alpha_i+1} \) given \( \alpha \) we can simply compute \( \copylen{\beta} \) on \( I\restr{\alpha_i+1} \supseteq I\restr{\alpha} \) for an appropriate \( i \).

To compute \( \enumpred{\beta} \) on \( I\restr{\alpha_i+1} \) we start listing \( \lambda \kleeneleq \alpha_i  \) and look for a \( \lambda \) and integer \( n \) such that \( \beta=\kleenelim{\lambda}{n} \).  If such a pair is ever found we return \( \enumpred{\beta}=\lambda \).  Simultaneously we start listing the sequence \( \kappa_i \) defined from \( \alpha=\alpha_i \) and should we discover \( \kappa_i=\beta \) we return \( \enumpred{\beta}=\diverge \).  The arguments given above guarantee that this is both a correct and complete procedure.  
	\end{proof}

We now must prove that there is a nice path through \( \kleeneO \).  We start by showing that we can computably build nice ordinal notations from arbitrary ordinal notations.

\begin{equation}
		\vec{\alpha} <_L \vec{\beta} \iffdef \vec{\alpha} \supsetneq \vec{\beta} \lor \exists({l})( \vec{\alpha}\restr{l}=\vec{\beta}\restr{l} \land \vec{\alpha}(l) \kleenel \vec{\beta}(l) )
\end{equation}

\begin{lemma}\label{lem:enum-sequence}
	Given any \( \alpha \) there is a computable procedure terminating on all \( \beta \kleeneleq \alpha \) yielding a finite sequence of notations \( \enumseq{\beta}=\str{\beta_0, \beta_1, \ldots, \beta_k} \) with \( \alpha= \beta_0 \kleeneg \beta_1 \kleeneg \ldots \kleeneg \beta_k=\beta \) such that 
	\begin{equation}
		\beta \kleenel \gamma \iff \enumseq{\beta} <_L \enumseq{\gamma}  \label{lem:enum-sequence:lex}
	\end{equation}
\end{lemma}
\begin{proof}
We wish to let \( \enumseq{\beta} \) control what sequences \( \vec{\varepsilon} \supset \enumseq{\beta}   \) will be associated with some notation so we tag each \( \enumseq{\beta} \) with a lower bound for notations appearing in \( \vec{\varepsilon} \supset \enumseq{\beta} \) when enumerated.

We start by enumerating \( \enumseq{\alpha}=\str{\alpha} \) and assigning \( \enumseq{\alpha} \) the lower bound \( 0 \).  Suppose we have already enumerated some \( \enumseq{\beta}=\str{\ldots, \beta} \) with associated lower bound \( \lambda \) but have yet to enumerate any extension of \( \enumseq{\beta} \).  Here we search for the least \( n \) such that  \(\kleenelim{\beta}{n} \kleenegeq \lambda \).  Let \( \gamma \) be the least notations such that \( \gamma +m = \kleenelim{\beta}{n} \) for finite \( m \) and \( \gamma \kleenegeq \lambda \) and enumerate \( \enumseq{\gamma}=\enumseq{\beta}\concat[\gamma] \) with lower bound \( \lambda \).  Otherwise let \( \kappa \) be the maximal notation with \( \enumseq{\kappa}=\enumseq{\beta}\concat[\kappa] \) already enumerated.  Now search for the least \( n \) with \(\kleenelim{\beta}{n} \kleenegeq \kappa \) and let \( \gamma \) be the least notations such that \( \gamma +m = \kleenelim{\beta}{n} \) for finite \( m \) and \( \gamma \kleeneg \kappa \) and enumerate \( \enumseq{\gamma}=\enumseq{\beta}\concat[\gamma] \) with lower bound \( \kappa + 1 \).

The construction clearly enumerates a sequence for every notation \( \beta \kleeneleq \alpha \).  Now assume that \( \enumseq{\beta} <_L \enumseq{\gamma} \).  Let \( \vec{\varepsilon}=\str{\ldots,\lambda} \) be the longest common initial segment of \( \enumseq{\beta}, \enumseq{\gamma} \) and \( \beta', \gamma' \) such that \( \vec{\varepsilon}\concat[\beta'] \subset \enumseq{\beta}  \) and  \( \vec{\varepsilon}\concat[\gamma'] \subset \enumseq{\gamma}  \).  During enumeration \( \vec{\varepsilon}\concat[\gamma'] \) would have been tagged with a lower bound of at least \( \beta' +1 \).  Hence \( \gamma \kleeneg \beta'  \) but \( \beta'  \kleenegeq \beta \).  Hence \( \beta \kleenel \gamma \).  To observe the other direction note that if \( \beta \kleenel \gamma \) either \( \enumseq{\beta} <_L \enumseq{\gamma} \) or \( \enumseq{\gamma} <_L \enumseq{\beta} \) but the later possibility would entail that \( \gamma \kleenel \beta \) so \( \enumseq{\beta} <_L \enumseq{\gamma} \). 
\end{proof}

\begin{lemma}\label{lem:nice-of-len-alpha}
	Given a notation \( \alpha \) we can effectively produce a nice notation \( \alpha' \) for the same ordinal.
\end{lemma}

\begin{proof}
	Given \( \beta \kleeneleq \alpha \) construct the sequence \( \enumseq{\beta} \) as by the prior lemma.  By transfinite recursion define the notation \( \hat{\beta} \) to be the successor of \( \hat{\gamma} \) if \( \beta \) is the successor of \( \gamma \) and the limit of \( \hat{\gamma_i} \) where \( \enumseq{\gamma_i}=\enumseq{\beta}\concat{\gamma_i} \) which is effective by the above construction.  Now if \( \kappa \kleenel \beta \) then \( \enumseq{\kappa} <_L \enumseq{\beta} \) and hence for some \( \gamma_i \)  \( \enumseq{\kappa} \leq_L \enumseq{\beta}\concat[\gamma_i]  \) so \( \kappa \kleeneleq \gamma_i \).  Thus, \( \lim_{i \to \infty} \gamma_i = \beta \).  Moreover, note that there is a path from \( \lambda' \) to \( \beta' \) only if \( \lambda \) appears in \( \enumseq{\beta} \) and that path is unique and minimal by construction satisfying part \ref{def:nice-ordinal:min} of the definition.  If \( \kleenelim{\lambda'}{n} \kleenel \beta \kleeneleq \kleenelim{\lambda'}{n+1} \) then \( \enumseq{\kleenelim{\lambda}{n}} <_L \enumseq{\beta} \leq_L \enumseq{\kleenelim{\lambda}{n+1}} \) so \( \enumseq{\lambda} \subseteq \enumseq{\beta} \) ensuring that part \ref{def:nice-ordinal:between} of the definition is satisfied.  Finally part \ref{def:nice-ordinal:finite-sequence} is trivially satisfied as for each \( \beta \) \( \enumseq{\beta} \) is a finite string. 
\end{proof}

\begin{lemma}\label{lem:nice-path}
There is a nice set \( I \) forming a path through \( \kleeneO \)
\end{lemma}

\begin{proof}
Fix a unique path \( \kleeneO* \) through \( \kleeneO \) and let \( \alpha_i \) be an increasing cofinal sequence in \( \kleeneO* \).  Define \( \kappa_0=\alpha'_0 \) and \( \kappa_{i+1}=\kappa_i + 1 + \alpha'_{i} \).  By the definition of the effective addition operation on notations there is no \( \lambda \kleenel \kappa_{i+1}  \) with \( \lambda \kleenegeq \kappa_i +1 \) and \( \beta \kleeneleq \kappa_i \) connected by a path but the set of ordinals therefore it follows from the fact that  \( \alpha'_i \) and \( \kappa_i \) are nice that \( \kappa_{i+1} \) is nice.  Let \( I = \set{\beta}{\exists({i})\left(\beta \kleenel \kappa_i \right)} \).
\end{proof}

To simplify our notation slightly in the main body of the paper we've made use of the fact that if \( \enumpred{\beta} \) is undefined then \( \copylen{\beta+1}=0 \) so we may safely set \( \enumpred{\beta}=\enumpred{(\beta+1)} \) for any \( \beta \) on which \( \enumpred{\beta} \) is undefined and by lemma \ref{lem:build-copylen} can be done without imperiling the computability of \( \copylen{\beta} \).    Note our construction of our nice path through \( \kleeneO \) provides a \( \Pi^1_1 \) set of notations \( \alpha_i \) cofinal in \( \kleeneO \) such that the computations giving \( \enumpred{\beta} \) and \( \copylen{\beta} \) for every \( \beta \kleeneleq \alpha_i \) can be uniformly computed from \( \alpha_i \).  All constructions performed in the main body of the paper can be taken to use ordinals that lie along this \( \Pi^1_1 \) path.

\end{document}